\documentclass{article}

\usepackage{zharticle}
\usepackage{graphicx,amsmath,amsfonts,amssymb,amsthm,mathtools,verbatim,booktabs}
\usepackage{enumitem,microtype}

\numberwithin{equation}{section}
\allowdisplaybreaks

\usepackage{amsthm}
\newtheoremstyle{zhthm}
  {\topsep}   
  {\topsep}   
  {}  
  {0pt}       
  {\sffamily\bfseries} 
  {}         
  {5pt plus 1pt minus 1pt} 
  {}          
\theoremstyle{zhthm}
\newtheorem{theorem}{Theorem}[section]
\newtheorem{proposition}[theorem]{Proposition}

\newtheorem{corollary}[theorem]{Corollary}

\newcounter{algorithm}
\renewcommand{\thealgorithm}{\arabic{algorithm}}
\newenvironment{algorithm}[1]{%
    \refstepcounter{algorithm}%
    \begin{list}{}{%
		\setlength{\rightmargin}{0\linewidth}%
		\setlength{\leftmargin}{.05\linewidth}%
    }%
		\small%
		\item[]{\setlength{\parskip}{0ex}\hrulefill\par\nopagebreak%
        {\sffamily\bfseries Algorithm~\thealgorithm}\ {\scshape#1}}%
        \begin{itemize}[label={}, leftmargin=0pt]%
}{%
    \end{itemize}%
    \vspace{-\parskip}
    {\setlength{\parskip}{-1ex}\nopagebreak\par\hrulefill}%
    \end{list}
}

\newcommand{\ones}{{\bf 1{}}}  
\newcommand{\reals}{{\bf R{}}}  
\newcommand{\diag}{\mathop{\bf diag}}  
\newcommand{\argmin}{\mathop{\rm argmin}}  
\newcommand{\ri}{\mathop{\bf relint}}  
\newcommand{\norm}[1]{{\lVert #1 \rVert}}  
\newcommand{\cK}{{\cal K{}}}
\newcommand{\fset}{{\cal F{}}}

\newcommand{\eg}{{\rmfamily\itshape e.g.}}
\newcommand{\ie}{{\rmfamily\itshape i.e.}}
\newcommand{\etal}{{\rmfamily\itshape et~al.}}

\usepackage{lipsum}

\usepackage[colorlinks]{hyperref}

\usepackage{listings}
\usepackage{xcolor}

\definecolor{codegreen}{rgb}{0,0.6,0}
\definecolor{codegray}{rgb}{0.5,0.5,0.5}
\definecolor{codepurple}{rgb}{0.58,0,0.82}
\definecolor{backcolour}{rgb}{0.95,0.95,0.98}

\lstdefinestyle{zhstyle}{
    backgroundcolor=\color{backcolour},
    commentstyle=\color{codegreen},
    keywordstyle=\color{magenta},
    numberstyle=\tiny\color{codegray},
    stringstyle=\color{codepurple},
    basicstyle=\ttfamily\small,
    breakatwhitespace=false,
    breaklines=true,
    captionpos=b,
    keepspaces=true,
    numbers=left,
    numbersep=5pt,
    showspaces=false,
    showstringspaces=false,
    showtabs=false,
    tabsize=2,
}
\lstdefinelanguage{zhpython}{%
    sensitive=true,%
    morekeywords={and, as, assert, break, class, continue, def, del, elif,%
        else, except, exec, finally, for, from, global, if, import, in, is,%
        lambda, not, or, pass, print, raise, return, try, while, with, yield},%
    morecomment=[l]\#,%
    morestring=[s]{'''}{'''},%
    morestring=[s]{"""}{"""},%
    morestring=[b]',%
    morestring=[b]"%
}

\title{Disciplined Bilevel Programming}
\author[1,2]{Hao Zhu}
\author[1,2]{Joschka Boedecker}
\affil[1]{IMBIT//BrainLinks-BrainTools}
\affil[2]{Department of Computer Science, University of Freiburg}

\begin{document}
\maketitle

\begin{abstract}
Bilevel optimization provides a natural modeling language for hierarchical decision problems.
However, applying existing numerical solvers usually requires substantial manual analysis and reformulation.
In this paper, we introduce \emph{disciplined bilevel programming} (DBLP), a symbolic framework that allows users to specify and solve optimistic bilevel problems in a high-level, human-readable way that is close to the mathematical formulation.
For problems with a disciplined nonlinear upper problem and a convex lower problem satisfying the disciplined parameterized programming rules, DBLP automatically canonicalizes the lower problem into conic form and constructs an equivalent single-level reformulation using the conic Karush-Kuhn-Tucker conditions.
We relax the resulting complementarity constraint and use a gap continuation procedure to approximately solve a sequence of smooth nonlinear problems.
We implement DBLP in the open-source Python package BLVPY, an extension of CVXPY for bilevel programming.
We demonstrate the modeling and solution capabilities of BLVPY on a range of bilevel optimization problems from several application domains.
The proposed framework and implementation allow users to specify and solve bilevel optimization problems within a few lines of code, without prior expertise in bilevel modeling and numerical optimization.
\end{abstract}

\newpage
\tableofcontents
\newpage

\section{Introduction}
We consider \emph{bilevel optimization problems} of the form
\begin{equation}\label{prob:blv}
    \begin{array}{ll}
        \mbox{minimize} & F_0(x, y)\\
        \mbox{subject to} & F_i(x, y) \leq 0, \quad i = 1, \ldots, m\\
        & y \in S(x),
    \end{array}
\end{equation}
where for a fixed $x \in \reals^n$, the constraint set $S(x)$ is defined as the solution set of the following \emph{lower problem} (or \emph{inner problem}):
\begin{equation}\label{prob:blv_lower}
    \begin{array}{rl}
    S(x) = \argmin_z & f_0(x, z)\\
    \mbox{subject to} & f_i(x, z) \leq 0, \quad i = 1, \ldots, p.
    \end{array}
\end{equation}
In this context, the problem (\ref{prob:blv}) is referred to as the \emph{upper problem} (or \emph{outer problem}), with $x \in \reals^n$ and $y \in \reals^k$ denoting the \emph{upper variable} and \emph{lower variable}, respectively.
Accordingly, the lower problem (\ref{prob:blv_lower}) is parameterized by the upper variable $x$.
The functions $F_0 \colon \reals^n \times \reals^k \to \reals$ and $f_0 \colon \reals^n \times \reals^k \to \reals$ are the upper- and lower-level objective functions, respectively, while $F_i$, $i = 1, \ldots, m$, and $f_i$, $i = 1, \ldots, p$, are their corresponding constraint functions.

Following a common convention in bilevel optimization~\cite{sinha2018review}, the generic formulations (\ref{prob:blv}) and (\ref{prob:blv_lower}) display only inequality constraints; equality constraints are also permitted in practice, but are omitted here for notational brevity.
Additionally, as shown in Shen~\etal~\cite{shen2016disciplined}, an equality constraint $F(x, y) = 0$ may instead be expressed indirectly as the pair of inequality constraints
\begin{equation*}
    F(x, y) \leq 0\quad\mbox{and}\quad -F(x, y) \leq 0.
\end{equation*}
We assume that the problem (\ref{prob:blv}) has \emph{optimistic} semantics, meaning that if the lower problem (\ref{prob:blv_lower}) has multiple optimal points (\ie, the set $S(x)$ is nonsingleton), then the upper problem (\ref{prob:blv}) chooses $y \in S(x)$ that is most favorable to the upper objective $F_0(x, y)$.
In the most general case, the bilevel problem (\ref{prob:blv}) is a nonconvex optimization problem that can be very difficult to solve, even if all involved functions $F_i$ and $f_i$ are convex.

\subsection{Previous and related work}\label{sec:prior_work}
\subsubsection{Hierarchical optimization}
Hierarchical optimization originates in Stackelberg games~\cite{stackelberg1952grundlagen} and was formulated as a mathematical program with an optimization problem in its constraints by Bracken and McGill~\cite{bracken1973mathematical}.
Bilevel models arise naturally when one decision maker must anticipate another's optimal response, with applications spanning economics, engineering, machine learning, and operations research.
For example, in transportation network design and toll setting, an authority chooses capacities or prices at the upper level, while travelers select routes at the lower level, often according to a traffic equilibrium; see, \eg, Bard~\cite[chapter 10]{bard1998practical} and Colson~\etal~\cite{colson2007overview}.
In energy and other regulated markets, a planner, regulator, or producer makes investment and pricing decisions while anticipating market clearing or the responses of consumers and competing firms~\cite[chapter 5]{dempe2020advances}.
In machine learning, the upper problem selects hyperparameters, regularization weights, or model architecture to minimize validation loss, whereas the lower problem trains the model on the training data~\cite{pedregosa2016hyperparameter,franceschi2018bilevel,lorraine2020optimizing,barratt2021least}; see also Dempe and Zemkoho~\cite[chapter 6]{dempe2020advances}.
Just to mention a few.

The distinction between optimistic and \emph{pessimistic} formulations of bilevel problems is essential when the lower solution mapping is set-valued:
The former permits an upper-favorable lower optimizer, whereas the latter protects against an upper-adverse one.
For more discussion about pessimistic bilevel problems, see Loridan and Morgan~\cite{loridan1996two} and Dempe~\cite{dempe2002foundations}.

There are many excellent books and surveys on bilevel optimization.
Bard~\cite{bard1998practical} gives a practical treatment of continuous, discrete, convex, and general bilevel programs, together with algorithms and applications in transportation, production planning, and pricing.
Dempe~\cite{dempe2002foundations} develops the theoretical foundations through parametric optimization, optimality conditions, stability and complexity analysis, and a detailed discussion of nonunique lower-level solutions.
Colson~\etal~\cite{colson2007overview} provide a concise survey of applications and solution methods and explain the connection between bilevel programs and mathematical programs with equilibrium constraints.
Sinha~\etal~\cite{sinha2018review} broaden the algorithmic review from classical methods to evolutionary approaches and discuss representative applications.
The edited volume of Dempe and Zemkoho~\cite{dempe2020advances} collects more recent developments in local, global, and heuristic algorithms, as well as extensions and applications including optimal control, energy markets, data analytics, and security.

\subsubsection{Solution methods}
In the most general case, finding an exact solution to a bilevel optimization problem of the form (\ref{prob:blv}) is NP-hard~\cite{hansen1992new,sinha2018review}, so practitioners usually rely on reformulations, approximations, or heuristic methods that do not guarantee global optimality to find an acceptable approximate solution.
We discuss two widely used approaches in the following paragraphs.
Descriptions of other methods, such as evolutionary algorithms and branch-and-bound methods for discrete bilevel programming, can be found in the literature listed above.

\paragraph{KKT single-level reformulations.}
When the lower problem (\ref{prob:blv_lower}) is convex and an appropriate constraint qualification holds, its Karush-Kuhn-Tucker (KKT) conditions replace lower optimality by primal feasibility,
stationarity, dual feasibility, and complementarity.
The resulting single-level problem is a mathematical program with \emph{equilibrium} (or specifically, \emph{complementarity}) constraints, and is exactly equivalent to the original bilevel problem (\ref{prob:blv}).
This classical connection and its stationarity theory are developed in \cite{ye1995optimality,luo1996mpec,scheel2000mpcc,dudda2006bilevel,dempe2006optimality,dempe2012reformulations} (just to list a few).
However, exact complementarity causes standard nonlinear programming constraint qualifications to fail, even when the original lower problem is regular.
Regularization and smoothing methods therefore solve nearby nonlinear programs and drive a relaxation parameter to zero~\cite{facchinei1999smoothing,scholtes2001regularization} (see also \S\ref{sec:approx_opt} for more details).

\paragraph{Differentiable optimization.}
Modern machine learning often treats bilevel problems by differentiating the upper objective through the procedure used to solve the lower problem.
One approach fixes a finite sequence of lower-level iterations and applies the chain rule through their update maps~\cite{maclaurin2015gradient,franceschi2018bilevel}.
Another assumes that the lower solution is locally unique and differentiates its first-order optimality equations using the implicit function theorem~\cite{pedregosa2016hyperparameter,lorraine2020optimizing}.
Closely related work derives and implements derivatives of solutions to parameterized optimization problems~\cite{gould2016differentiating,amos2017optnet,agrawal2019differentiable,agrawal2019cone,healey2026differentiating}.
In both approaches, the derivative depends on a particular choice of lower-level solution, \ie, either the result of a fixed number of algorithmic steps or a locally unique solution that varies smoothly with the upper variables.
Moreover, most differentiable optimization methods mentioned above assume an unconstrained upper problem or only simple constraints on the upper variables, while general constraints involving the lower solution, as in the problem (\ref{prob:blv}), are usually not supported.

\subsubsection{Limitations}
Despite the widely recognized applications and solution methods for bilevel optimization, a significant gap remains between natural problem formulations (\eg, of the form (\ref{prob:blv})) and the \emph{canonical} forms required by existing solution methods.
For example, a KKT reformulation first requires verifying lower-level convexity and a suitable constraint qualification.
It then introduces a dual variable and a complementarity condition for each lower-level inequality, which must be encoded in a form accepted by the chosen solver~\cite[\S2 and \S4]{dias2024bileveljump}.
Carrying out these steps correctly requires substantial expertise and is often tedious and error-prone, even for experienced practitioners.

This gap has been partially addressed by different \emph{domain specific languages} (DSLs) for bilevel optimization in different programming languages.
The early Python package PAO~\cite{hart2019pao} extends Pyomo~\cite{hart2011pyomo,bynum2021pyomo} to support the symbolic modeling and solving of multilevel optimization problems, but is no longer actively updated or maintained.
The more recent Julia package BilevelJuMP.jl~\cite{dias2024bileveljump} extends JuMP~\cite{dunning2017jump,lubin2023jump} to support bilevel optimization with conic constraints and quadratic objectives in the lower problem, but it does not support nonlinear expressions at that level.
As a result, users must still manually reformulate their problems into at least the canonical conic forms required by BilevelJuMP.jl, which is not ideal for users not well-versed in convex optimization.

\subsection{Contributions}
In this paper, we introduce \emph{disciplined bilevel programming} (DBLP), a symbolic modeling framework for bilevel optimization problems that allows users to specify and solve bilevel problems in a natural, human-readable way that is very close to the original mathematical formulation (\ref{prob:blv}).
Given a compatible symbolic bilevel problem specification, DBLP allows automatic canonicalization of the problem into a standard single-level optimization problem that can be solved by existing numerical solvers for general nonlinear optimization.
The associated DSL is implemented in the Python package BLVPY, which is fully open-source and available at
\begin{quote}
    \url{https://github.com/dxogrp/blvpy}.
\end{quote}

DBLP is based on the ideas of \emph{disciplined optimization}, whose principles consist in allowing users to construct optimization problems from a set of basic functions and composition rules, while guaranteeing that the resulting problem is compatible with a specific class of solvers (see \S\ref{sec:dsopt} for more discussion and references).
We adapt the KKT reformulation of bilevel problems to the disciplined optimization framework, and propose a rewriting procedure that transforms a DBLP-compliant bilevel problem into a single-level optimization problem compatible with general nonlinear solvers.
As a result, users can specify and solve bilevel optimization problems in short scripts without the expert knowledge of mathematical analysis or optimization.
As far as we know, there is no existing work that combines bilevel optimization with the disciplined optimization philosophy.

We should also note that the purpose of this work is \emph{not} to propose a new solution method or theoretical results for bilevel optimization, but to fill the gap between the original mathematical problem formulation and the canonical forms required by a numerical solver, based on many existing results in the literature.

\subsection{Outline}
The rest of this paper is organized as follows.
In \S\ref{sec:dsopt}, we provide a short review of disciplined optimization, focusing on disciplined convex, parameterized, and nonlinear programming.
In \S\ref{sec:dblp}, we introduce the DBLP framework and develop the lower problem canonicalization and exact single-level reformulation via the KKT conditions.
In \S\ref{sec:approx_opt}, we present the relaxed complementarity condition, its approximate lower-level optimality guarantees, and the gap continuation procedure for (approximately) solving the reformulated problem.
The BLVPY implementation of the proposed DBLP framework and its basic usage are described in \S\ref{sec:impl}.
Finally, in \S\ref{sec:example}, we present numerical examples of applications appearing in different domains, followed by some discussion and comments in \S\ref{sec:discussion}.

\section{Disciplined optimization}\label{sec:dsopt}
In this section, we provide a short introduction to disciplined optimization and present some of its key concepts and results.

\subsection{Disciplined convex programming}
The original idea of disciplined optimization was proposed by Grant~\etal~\cite{grant2005disciplined,grant2006disciplined} in the context of convex optimization, and is referred to as \emph{disciplined convex programming} (DCP).
The fundamental philosophy of DCP is that, rather than constructing objective and constraint functions without regard for convexity, users draw from a library of \emph{atomic functions} with known convexity properties and combine them in ways which the convexity preserving operations insure convex outcomes.

Specifically, a DCP problem has the form
\begin{equation}\label{prob:dcp}
    \begin{array}{ll}
        \mbox{minimize/maximize} & f(x)\\
        \mbox{subject to} & p_i(x) \sim q_i(x),\quad i = 1, \ldots, m,
    \end{array}
\end{equation}
where $x \in \reals^n$ is the variable, the function $f \colon \reals^n \to \reals$ is the objective, and $p_i, q_i \colon \reals^n \to \reals$ are the left-hand side and right-hand side functions of the $i$th constraint.
The relational operator $\sim$ represents one of the relations $\leq$, $\geq$, or $=$.
In DCP this problem must be convex, which imposes the following additional rules on the curvature of the involved functions in the problem (\ref{prob:dcp}) listed below.
\begin{itemize}
    \item For a minimization problem, the objective function $f$ must be convex; for a maximization problem, $f$ must be concave.
    \item For a constraint of the form $p_i(x) \leq q_i(x)$, the function $p_i$ must be convex, and $q_i$ must be concave.
    \item For a constraint of the form $p_i(x) \geq q_i(x)$, the function $p_i$ must be concave, and $q_i$ must be convex.
    \item For a constraint of the form $p_i(x) = q_i(x)$, both $p_i$ and $q_i$ must be affine.
\end{itemize}
To verify the convexity of a DCP problem in the form (\ref{prob:dcp}), each function (or \emph{expression}) in the problem is parsed into a \emph{composition tree}, whose nodes represent functions with known convexity or convexity preserving operations (see, \eg, \cite[\S3.2]{boyd2004convex}).
We (or a computer) can therefore verify the convexity of the objective and constraints constructively by traversing their composition trees from the leaves to the root and checking whether the resulting functions satisfy the curvature requirements listed above (see Grant~\etal~\cite{grant2006disciplined} for some specific examples).

The DCP framework was originally implemented in the MATLAB package CVX~\cite{grant2014cvx}.
It was later implemented by several other DSLs in different programming languages, including Convex.jl~\cite{udell2014convex} for Julia, CVXPY~\cite{diamond2016cvxpy,agrawal2018rewriting,zhang2026cvxpy} for Python, CVXR~\cite{fu2020cvxr} for R, and CVXRust~\cite{diamond2025cvxrust} for Rust.
Given a DCP-compliant convex program specified in any of these DSLs, one of the most important practical implications is that the problem can be automatically canonicalized.
That is, a user only needs to specify the target problem in a high-level expression close to the mathematical formulation, and the DSL modeling system will automatically convert it into a standard form accepted by general convex conic solvers~\cite{agrawal2018rewriting}.
Finally, a compatible numerical solver, such as SCS~\cite{odonoghue2016conic}, Clarabel~\cite{goulart2024clarabel}, or MOSEK~\cite{mosek2026mosek}, is called to solve the canonicalized problem.

Our work here extends the modeling language CVXPY to support bilevel programming.
Although CVXPY has been extended over the years to support many types of nonconvex optimization problems, such as \emph{stochastic}~\cite{ali2015disciplined}, \emph{convex-concave}~\cite{shen2016disciplined}, \emph{multi-convex}~\cite{shen2017disciplined}, \emph{geometric}~\cite{agrawal2019disciplined}, \emph{quasiconvex}~\cite{agrawal2020disciplined}, \emph{saddle}~\cite{schiele2024disciplined}, and \emph{biconvex}~\cite{zhu2026disciplined} optimization problems, to the best of our knowledge, disciplined modeling of bilevel optimization has not been implemented in any of these DSLs.

\subsection{Problems involving parameters}
One of the key extensions of the original DCP framework from Grant~\etal~\cite{grant2006disciplined} and Diamond~\etal~\cite{diamond2016cvxpy} that is closely related to bilevel optimization is the concept of \emph{disciplined parameterized programming} (DPP)~\cite{agrawal2019differentiable}.
DPP deals with convex optimization problems that involve \emph{parameters}, which are variables whose values are fixed when the problem is solved, but can be changed between solves.
For example, the lower problem (\ref{prob:blv_lower}) is a parameterized optimization problem with parameter $x \in \reals^n$.

DPP imposes additional restrictions on the original DCP ruleset.
Thus, every DPP-compliant problem is also DCP-compliant, but the converse does not hold.
Specifically, DPP introduces two restrictions on the involved functions, or expressions, of a parameterized convex problem:
\begin{enumerate}
    \item In DCP, all parameters themselves are classified as constants, while in DPP, parameters are classified as affine expressions.
    \item In DCP, the product expression $\phi(u, v) = uv$ is affine if either $u$ or $v$ is constant (\ie, does not depend on the variables), while in DPP, the product expression $\phi(u, v) = uv$ is affine if
        \begin{itemize}
            \item either subexpression $u$ or $v$ is constant (\ie, depends on neither the parameters nor the variables), or
            \item one of the expressions $u$ and $v$ is \emph{parameter-affine} (\ie, does not involve variables and is affine in its parameters) and the other is \emph{parameter-free} (\ie, does not involve parameters).
        \end{itemize}
\end{enumerate}
The rules above describe the basic DPP restrictions, but they can be (and actually, have been) extended to cover additional combinations of expressions and parameters~\cite{zhang2026cvxpy}.

As DCP generates programs that are guaranteed to be convex, DPP generates convex programs that are additionally guaranteed to be reducible to \emph{affine-solve-affine} (ASA) compatible canonical forms~\cite{agrawal2019differentiable}.
This means that, for a DPP-compliant problem, both the map from the original problem parameters to the canonical problem data and the map from the canonical solution back to the original solution are affine.
Furthermore, the coefficients of these affine maps are constant, independent of the specific values of the problem parameters.

\subsection{Nonlinear optimization problems}
Disciplined nonlinear programming (DNLP) extends the atom-and-composition approach of DCP to nonlinear optimization problems that need not be convex~\cite{cederberg2026disciplined}.
It allows arbitrary smooth expressions and classifies certain nonsmooth expressions as \emph{linearizable convex} or \emph{linearizable concave} using composition rules analogous to those of DCP\@.
Specifically, a DNLP-compliant problem satisfies the following rules:
\begin{itemize}
    \item For a minimization problem, the objective must be linearizable convex; for a maximization problem, it must be linearizable concave.
    \item For a constraint of the form $p_i(x) \leq q_i(x)$, the function $p_i$ must be linearizable convex, and $q_i$ must be linearizable concave.
    \item For a constraint of the form $p_i(x) \geq q_i(x)$, the function $p_i$ must be linearizable concave, and $q_i$ must be linearizable convex.
    \item For a constraint of the form $p_i(x) = q_i(x)$, both $p_i$ and $q_i$ must be smooth.
\end{itemize}

A DNLP-compliant problem can be canonicalized into an equivalent smooth nonlinear program.
During this process, admissible nonsmooth convex or concave subexpressions are replaced by auxiliary variables and appropriate epigraph or hypograph constraints, which are then expressed using smooth functions.
For example, the epigraph constraint $t \geq |u|$ can be written as the pair of smooth inequalities $-t \leq u \leq t$.
Finally, a canonicalized smooth nonlinear program is obtained, which can be solved by standard nonlinear solvers such as IPOPT~\cite{wachter2006implementation}, KNITRO~\cite{byrd2006knitro}, COPT~\cite{ge2022cardinal}, and Uno~\cite{vanaret2026implementing,vanaret2026uno}.

This DCP-style canonicalization implemented in DNLP has several practical advantages.
It allows nonsmooth modeling constructs to be handled by standard smooth nonlinear solvers, simplifies the construction of a valid initial point, and avoids reformulations that unnecessarily violate regularity conditions such as the linear independence constraint qualification.
These features generally improve the robustness of the nonlinear optimization solution process and the quality of the resulting solution~\cite{cederberg2026disciplined}.
Unlike DCP, however, DNLP does not certify convexity or global optimality.

\section{Disciplined bilevel programming}\label{sec:dblp}
We now present the main ideas of DBLP, based on the concepts of disciplined optimization introduced in the previous section.
We define a \emph{disciplined bilevel program} as a bilevel optimization problem of the form (\ref{prob:blv}) that satisfies the following rules:
\begin{itemize}
    \item The objective and constraint functions $F_i \colon \reals^n \times \reals^k \to \reals$ for $i = 0, 1, \ldots, m$ of the upper problem (\ref{prob:blv}) are DNLP-compliant with variables $x \in \reals^n$ and $y \in \reals^k$.
    \item The objective and constraint functions $f_i \colon \reals^n \times \reals^k \to \reals$ for $i = 0, 1, \ldots, p$ of the lower problem (\ref{prob:blv_lower}) are DPP-compliant with variable $z \in \reals^k$, so that the lower problem is a disciplined convex program, parameterized by $x \in \reals^n$.
\end{itemize}
These two rules allow us to conduct the following canonicalization of the bilevel problem (\ref{prob:blv}) into a single-level optimization problem.

\subsection{Lower problem canonicalization}
Suppose the lower problem (\ref{prob:blv_lower}) is a DPP-compliant convex program with parameter $x \in \reals^n$.
Then this problem can be canonicalized into a standard conic form
\begin{equation}\label{prob:lower_cone}
    \begin{array}{ll}
        \mbox{minimize} & {c(x)}^T u + d(x)\\
        \mbox{subject to} & A(x)u + s = b(x)\\
        & s \in \cK,
    \end{array}
\end{equation}
where the vectors $u$ and $s$ are the variables, and $c(x)$, $d(x)$, $A(x)$, and $b(x)$ are the problem data.
The cone $\cK$ is a Cartesian product of proper convex cones.
Additionally, by Agrawal~\etal~\cite{agrawal2019differentiable}, the map from the original problem parameter $x$ to the group of canonical problem data, given by
\begin{equation*}
    x \mapsto (c(x), d(x), A(x), b(x)),
\end{equation*}
is affine, with coefficients that are independent of the specific value of $x$.
Note that the offset $d(x)$ in the objective of the canonical form (\ref{prob:lower_cone}) is a constant term which does not affect the optimal point of the lower problem (\ref{prob:blv_lower}).
We include it here for notational simplicity in the following derivations.

\subsection{Standing assumptions}\label{sec:asps}
Before proceeding with the canonicalization of the bilevel problem (\ref{prob:blv}) via the conic formulation (\ref{prob:lower_cone}), we state the following convenient sufficient assumptions, which apply throughout the remainder of this paper:
\begin{itemize}
    \item The admissible domain of the lower-level parameter $x$ is compact.
        The functions $F_i$ and $f_i$ in (\ref{prob:blv}), and $c$, $d$, $A$, and $b$ in (\ref{prob:lower_cone}) are continuous.
        The set defined by the upper constraints of (\ref{prob:blv}) is closed.
    \item The canonicalization is \emph{lossless}, \ie, it preserves both the optimal value and the complete solution set through a fixed affine recovery map $\rho$.
        In other words, let $p^\star(x)$ and $q^\star(x)$ be the optimal values of the original lower problem (\ref{prob:blv_lower}) and its canonical form (\ref{prob:lower_cone}), respectively.
        Then for every feasible $x$, we have
        \begin{equation*}
            p^\star(x)=q^\star(x)\quad\mbox{and}\quad
            S(x)=\left\{\rho(u)\ \middle|\ \begin{array}{l}
                {c(x)}^T u+d(x)=q^\star(x)\\
                A(x)u+s=b(x)\\
                s\in\cK\\
            \end{array}\right\}.
        \end{equation*}
        Moreover, every feasible $(u,s)$ yields a feasible lower point $\rho(u)$ satisfying
    \begin{equation*}
        f_0(x,\rho(u)) \leq {c(x)}^T u+d(x).
    \end{equation*}
    \item For every feasible $x$, the lower problem has a finite attained optimum.
    \item For every feasible $x$, the canonical problem satisfies Slater's condition, \ie, there exist $\tilde{u}$ and $\tilde{s}$ such that
    \begin{equation*}
        A(x)\tilde{u}+\tilde{s}=b(x), \quad \tilde{s} \in \ri \cK,
    \end{equation*}
    where $\ri C$ denotes the relative interior of the set $C$.
\end{itemize}
Since $\cK$ is a Cartesian product of closed convex cones, the last two assumptions imply strong duality and dual attainment for the canonical problem (\ref{prob:lower_cone})~\cite[\S5.9]{boyd2004convex}.

The aforementioned technical conditions are standard in parametric and bilevel optimization and are expected to hold for the intended applications~\cite[chapter 4]{dempe2002foundations}.
Our implementation (see \S\ref{sec:impl}) verifies the structural conditions that can be checked symbolically, including DPP and DNLP compliance and the use of a fixed, audited reduction path for lossless canonicalization~\cite{agrawal2018rewriting,agrawal2019differentiable,cederberg2026disciplined,zhang2026cvxpy}; the remaining analytic conditions, such as attainment and Slater's condition, are treated as user-asserted preconditions (which are, nevertheless, satisfied in most practical applications involving convex lower problems).

\subsection{Conic optimality}
We define the \emph{dual cone} of $\cK$ in the (primal) problem (\ref{prob:lower_cone}) as
\begin{equation*}
    \cK^* = \{\lambda \mid \mbox{$\lambda^T s \geq 0$ for all $s \in \cK$}\}.
\end{equation*}
Then the dual problem of the canonical conic lower problem (\ref{prob:lower_cone}) is
\begin{equation*}
    \begin{array}{ll}
        \mbox{maximize} & -{b(x)}^T \lambda + d(x)\\
        \mbox{subject to} & {A(x)}^T \lambda + c(x) = 0\\
        & \lambda \in \cK^*,
    \end{array}
\end{equation*}
where $\lambda$ is the (dual) variable and $x$ is the parameter (see, \eg, \cite[page 266]{boyd2004convex}).
According to the assumptions stated in \S\ref{sec:asps}, we have the following strong duality result:
For every feasible $x$, a point $y$ solves the lower problem if and only if there exist $u$, $s$, and $\lambda$ satisfying
\begin{equation}\label{eq:kkt_conic}
    \begin{array}{rcl}
        y &=& \rho(u)\\
        A(x)u+s &=& b(x)\\
        {A(x)}^T \lambda + c(x) &=& 0\\
        s &\in& \cK\\
        \lambda &\in& \cK^*\\
        \lambda^T s &=& 0.
    \end{array}
\end{equation}
The first equation in (\ref{eq:kkt_conic}) recovers the original lower variable $y$ from the canonical variable $u$, while the remaining equations are the primal-feasibility, dual-feasibility, and complementarity conditions for the canonical lower problem (\ref{prob:lower_cone})~\cite[\S2.1]{odonoghue2016conic}.
In other words, the KKT conditions (\ref{eq:kkt_conic}) are necessary and sufficient for optimality of the lower problem (\ref{prob:blv_lower}), and therefore, the set of lower optimal points $S(x)$ can be equivalently expressed as
\begin{equation*}
    S(x) = \left\{\rho(u)\ \middle|\ \begin{array}{l}
        A(x)u+s=b(x)\\
        {A(x)}^T \lambda + c(x)=0\\
        s\in\cK, \quad \lambda\in\cK^*\\
        \lambda^T s=0
    \end{array}\right\}.
\end{equation*}

\subsection{Single-level reformulation}
A direct implication of the KKT conditions (\ref{eq:kkt_conic}) is that the bilevel problem (\ref{prob:blv}) can be equivalently expressed as a single-level optimization problem with variables $x$, $u$, $s$, and $\lambda$.
Specifically, the resulting problem is a nonlinear program with complementarity constraints~\cite{luo1996mpec,scheel2000mpcc}, given by
\begin{equation}\label{prob:single_level}
    \begin{array}{ll}
        \mbox{minimize} & F_0(x, \rho(u))\\
        \mbox{subject to} & F_i(x, \rho(u)) \leq 0, \quad i = 1, \ldots, m\\
        & A(x)u+s = b(x)\\
        & {A(x)}^T \lambda + c(x) = 0\\
        & s \in \cK,\quad \lambda \in \cK^*\\
        & \lambda^T s = 0,
    \end{array}
\end{equation}
where the problem variables are the original upper variable $x$, the canonical lower variable $u$, the slack variable $s$, and the dual variable $\lambda$.
We use the standard term \emph{lift} to describe an extended representation of a given optimization problem obtained by introducing auxiliary variables~\cite{gouveia2013lifts}.
Accordingly, a feasible point $(x, u, s, \lambda)$ of the single-level reformulation (\ref{prob:single_level}) lifts the corresponding point $(x, y)$ of the original bilevel problem (\ref{prob:blv}), with $y = \rho(u)$.
In this context, the single-level reformulation (\ref{prob:single_level}) is sometimes called a \emph{lifted problem} of (\ref{prob:blv}).

Recall that $c(x)$, $A(x)$, and $b(x)$ are affine in $x$, and that the solution recovery map $\rho$ is affine.
Consequently, the compositions $F_i(x,\rho(u))$ preserve the DNLP compliance (as well as the convexity) of the upper-level expressions, while the equality constraints introduced by the KKT reformulation are smooth.
The cone-membership constraints remain convex in $s$ and $\lambda$, while the complementarity equation $\lambda^T s=0$ is bilinear and nonconvex.
Thus, even when all $F_i$ are convex, the single-level reformulation is generally nonconvex because of complementarity and the parameter-variable products in the primal- and dual-feasibility equations.

Nevertheless, according to the discussion above, if the original bilevel problem (\ref{prob:blv}) is DBLP-compliant (according to the rules defined at the beginning of \S\ref{sec:dblp}), the single-level reformulation (\ref{prob:single_level}) is DNLP-compliant, and hence can be canonicalized into a smooth nonlinear program that can be solved by standard nonlinear solvers.

\subsection{A simple example}
As a minimal example, consider the bilevel problem
\begin{equation}\label{prob:example_bilevel}
    \begin{array}{ll}
        \mbox{minimize} & \norm{x-g}_2^2 + \gamma\norm{y-h}_2^2\\
        \mbox{subject to} & -r\ones \preceq x \preceq r\ones\\
        & y \in S(x)
    \end{array}
\end{equation}
with variables $x, y \in \reals^n$, where $g, h \in \reals^n$, $\gamma > 0$, and $r > 0$ are problem data, and
\begin{equation}\label{prob:example_lower}
    \begin{array}{rl}
    S(x) = \argmin_z & \norm{z - x}_1\\
    \mbox{subject to} & z \succeq 0.
    \end{array}
\end{equation}
All vector inequalities in this example are componentwise.

To rewrite the bilevel problem (\ref{prob:example_bilevel}) into the single-level form (\ref{prob:single_level}), we first canonicalize the lower problem (\ref{prob:example_lower}) into the conic form (\ref{prob:lower_cone}).
Introducing an epigraph variable $t \in \reals^n$, the lower problem can be written as
\begin{equation*}
    \begin{array}{ll}
        \mbox{minimize} & \ones^T t\\
        \mbox{subject to} & -t \preceq x - z \preceq t\\
        & z \succeq 0.
    \end{array}
\end{equation*}
Define the canonical variable $u$ and slack variable $s$ as
\begin{equation*}
    u = \left[\begin{array}{c}
        z \\ t
    \end{array}\right] \in \reals^{2n}
    \quad\mbox{and}\quad
    s = \left[\begin{array}{c}
        s_1\\s_2\\s_3
    \end{array}\right] \in \reals^{3n}.
\end{equation*}
The problem then has the canonical conic form
(\ref{prob:lower_cone}) with
\begin{equation*}
    c = \left[\begin{array}{c}
        0 \\ \ones
    \end{array}\right],\qquad
    d(x) = 0,\qquad
    A = \left[\begin{array}{cc}
         I & -I\\
        -I & -I\\
        -I & 0
    \end{array}\right],\qquad
    b(x) = \left[\begin{array}{c}
         x \\ -x \\ 0
    \end{array}\right],\qquad
    \cK = \reals_+^{3n},
\end{equation*}
and the solution recovery map is the coordinate projection
\begin{equation*}
    \rho(u) = \left[\begin{array}{cc}
        I & 0
    \end{array}\right] u = z.
\end{equation*}
Note that the canonical problem data and the recovery map are indeed affine in the parameter $x$ and canonical variable $u$, respectively.
Moreover, canonical feasibility is equivalent to $z\succeq0$ and $t\succeq|z-x|$, and therefore
$\ones^T t \geq \norm{z-x}_1$.
Conversely, every lower-feasible $z$ admits a canonical feasible lift with $t=|z-x|$, for which equality holds.
Thus, the canonicalization preserves the optimal value and complete solution set under the recovery map $\rho(u)=z$, and is lossless in the sense of the second assumption in \S\ref{sec:asps}.

Since the nonnegative orthant is self-dual, partition the dual variable
as
\[
    \lambda=\left[\begin{array}{c}
        \lambda_1\\\lambda_2\\\lambda_3
    \end{array}\right]\in\reals_+^{3n},
\]
conformably with $s$.
The dual of the canonical lower problem associated with (\ref{prob:example_bilevel}) is then given by
\begin{equation*}
    \begin{array}{ll}
        \mbox{maximize} & x^T(\lambda_2-\lambda_1)\\
        \mbox{subject to}
            & \lambda_1-\lambda_2-\lambda_3=0\\
            & -\lambda_1-\lambda_2+\ones=0\\
            & \lambda_1, \lambda_2, \lambda_3 \succeq 0.
    \end{array}
\end{equation*}
The canonical lower problem is strictly feasible for every $x$:
One may choose any $z\succ0$ and then choose $t$ sufficiently large.
It also has a finite attained optimum.
Consequently, strong duality holds and the lower-level condition $y\in S(x)$ is equivalent to the existence of $z,t,s_1,s_2,s_3,\lambda_1,\lambda_2,\lambda_3$ satisfying
\begin{equation}\label{eq:example_kkt}
    \begin{array}{rcl}
        y &=& z\\
        z-t+s_1 &=& x\\
        -z-t+s_2 &=& -x\\
        -z+s_3 &=& 0\\
        \lambda_1-\lambda_2-\lambda_3 &=& 0\\
        -\lambda_1-\lambda_2+\ones &=& 0\\
        s_1,s_2,s_3 &\succeq& 0\\
        \lambda_1,\lambda_2,\lambda_3 &\succeq& 0\\
        \lambda_1^T s_1+\lambda_2^T s_2+\lambda_3^T s_3 &=& 0.
    \end{array}
\end{equation}
The first equality in (\ref{eq:example_kkt}) identifies the recovered lower variable $y$ with $z$.
The next three equalities, together with $s_1,s_2,s_3\succeq0$, impose primal feasibility; the following two equalities, together with $\lambda_1,\lambda_2,\lambda_3\succeq0$, impose dual feasibility; and the final equality imposes complementarity.
Because every term in the final sum is nonnegative, this scalar condition is equivalent to componentwise complementarity between each $s_i$ and $\lambda_i$.

Finally, substituting $y=z$ and the remaining conditions in (\ref{eq:example_kkt}) for the lower-level constraint in (\ref{prob:example_bilevel}) gives the single-level problem
\begin{equation}\label{prob:example_single_level}
    \begin{array}{ll}
        \mbox{minimize}
            & \norm{x-g}_2^2+\gamma\norm{z-h}_2^2\\
        \mbox{subject to}
            & -r\ones\preceq x\preceq r\ones\\
            & z-t+s_1=x\\
            & -z-t+s_2=-x\\
            & -z+s_3=0\\
            & \lambda_1-\lambda_2-\lambda_3=0\\
            & -\lambda_1-\lambda_2+\ones=0\\
            & s \succeq 0,\quad \lambda \succeq 0\\
            & \lambda^T s=0,
    \end{array}
\end{equation}
with variables $x, z, t \in \reals^n$ and $s, \lambda \in \reals^{3n}$, which is equivalent to the original bilevel problem (\ref{prob:example_bilevel}).
The single-level reformulation (\ref{prob:example_single_level}) of the original bilevel problem (\ref{prob:example_bilevel}) is easily seen to be DNLP-compliant.
From this example, we can also see that performing the canonicalization and KKT reformulation by hand is tedious and error-prone, even for a simple bilevel optimization problem like the one presented.

\section{Approximate optimality and continuation}\label{sec:approx_opt}
Although the single-level reformulation (\ref{prob:single_level}) is DNLP-compliant, its exact complementarity constraint gives rise to a degenerate feasible geometry for which standard nonlinear programming constraint qualifications generally fail.
Consequently, directly applying off-the-shelf nonlinear solvers may lead to numerical difficulties and weakened convergence guarantees~\cite{scheel2000mpcc,scholtes2001regularization}.
As a result, we propose a continuation method that solves a sequence of approximate problems with relaxed complementarity constraints and gradually drives the relaxation parameter to zero.

\subsection{Relaxed complementarity}
We consider the following relaxation of the single-level reformulation (\ref{prob:single_level}) with a relaxation parameter $\epsilon > 0$:
\begin{equation}\label{prob:approx_single_level}
    \begin{array}{ll}
        \mbox{minimize} & F_0(x, \rho(u))\\
        \mbox{subject to} & F_i(x, \rho(u)) \leq 0, \quad i = 1, \ldots, m\\
        & A(x)u+s = b(x)\\
        & {A(x)}^T \lambda + c(x) = 0\\
        & s \in \cK,\quad \lambda \in \cK^*\\
        & \lambda^T s \leq \epsilon,
    \end{array}
\end{equation}
where the problem variables are the same as in (\ref{prob:single_level}).

The relaxation parameter $\epsilon$ in the problem (\ref{prob:approx_single_level}) has the following direct lower-level optimality interpretation.

\begin{proposition}\label{prop:lower_certificate}
    \emph{Lower-level objective certificate.}
    Under the assumptions in \S\ref{sec:asps}, every feasible point of (\ref{prob:approx_single_level}) satisfies
    \begin{equation}\label{eq:lower_certificate}
        0 \leq f_0(x,\rho(u))-p^\star(x)
        \leq \lambda^T s \leq \epsilon,
    \end{equation}
    where $f_0(x,\rho(u))$ is the lower-level objective evaluated at the recovered lower point $\rho(u)$, and $p^\star(x)$ is the optimal value of the lower problem (\ref{prob:blv_lower}) for the given parameter $x$.
\end{proposition}

The proof of this proposition is provided in \S\ref{sec:low_cert_proof} of the appendix.
This result shows that the actual complementarity gap $\lambda^T s$ is an a posteriori certificate of lower-level suboptimality, which may be sharper than the prescribed tolerance $\epsilon$. Thus, every feasible point of (\ref{prob:approx_single_level}) recovers a lower-feasible point that is at most $\epsilon$-suboptimal, and driving $\epsilon$ to zero enforces exact lower-level optimality in objective value.
Note that although proposition~\ref{prop:lower_certificate} controls lower-level objective suboptimality, it does not by itself guarantee either proximity of $\rho(u)$ to $S(x)$ or optimality with respect to the upper problem; strong convexity provides the former guarantee.

\begin{corollary}\label{cor:strongly_convex}
    \emph{Strongly convex lower problem.}
    Fix a feasible $x$ and suppose that $f_0(x, \cdot)$ is $\mu$-strongly convex on the feasible set of the lower problem, where $\mu > 0$.
    Then $S(x) = \{y^\star(x)\}$ is a singleton set, and every feasible point of (\ref{prob:approx_single_level}) satisfies
    \begin{equation}\label{eq:lower_distance_bound}
        \norm{\rho(u) - y^\star(x)}_2 \leq \sqrt{\frac{2\epsilon}{\mu}},
    \end{equation}
    where $y^\star(x)$ is the unique lower minimizer for the given parameter $x$.
\end{corollary}

The proof of this corollary is provided in \S\ref{sec:strongly_convex_proof} of the appendix.
Under the strong-convexity condition of corollary~\ref{cor:strongly_convex}, objective accuracy also controls solution accuracy, with distance decreasing at the rate $O(\sqrt{\epsilon})$.
In particular, to guarantee $\norm{\rho(u)-y^\star(x)}_2\leq\delta$, it suffices to choose $\epsilon\leq\mu\delta^2/2$.
These observations motivate the gap continuation procedure developed below.

\subsection{Gap continuation}\label{sec:gap_continuation}
We consider the following algorithm for solving the bilevel problem (\ref{prob:blv}) via a sequence of approximate problems (\ref{prob:approx_single_level}) with decreasing relaxation parameters $\epsilon \to 0$.
\begin{algorithm}{Gap continuation.}\label{alg:gap_cont}
    \item \textbf{given} initial point $(x^{(0)},u^{(0)},s^{(0)},\lambda^{(0)})$, initial relaxation parameter $\epsilon^{(0)}>0$, and reduction factor $\beta\in(0,1)$.
    \item $k \coloneqq 0$.
    \item \textbf{repeat}
        \begin{enumerate}
            \item \emph{Solution.} Approximately solve the relaxation (\ref{prob:approx_single_level}) with $\epsilon=\epsilon^{(k)}$ using $(x^{(k)},u^{(k)},s^{(k)},\lambda^{(k)})$ as the initial point, and denote the solution by $(x^{(k+1)},u^{(k+1)},s^{(k+1)},\lambda^{(k+1)})$.
            \item \emph{Gap reduction.} $\epsilon^{(k+1)} \coloneqq \beta \epsilon^{(k)}$.
            \item $k \coloneqq k + 1$.
        \end{enumerate}
    \item \textbf{until} $\epsilon^{(k)}$ is sufficiently small.
\end{algorithm}
Note that the algorithm uses the computed solution of each relaxation to initialize the next one.
Because consecutive relaxations differ only in the complementarity tolerance, gradually reducing $\epsilon$ often provides an effective warm start in practice.
Moreover, the initial point need not be feasible unless required by the selected nonlinear solver, although a feasible and well-scaled initial guess may improve numerical robustness.
Implementation details of algorithm~\ref{alg:gap_cont} regarding the initialization and other numerical aspects are discussed in \S\ref{sec:impl}.

For completeness, \S\ref{sec:gap_continuation_consistency} of the appendix establishes an idealized consistency result for globally solved relaxations.
This result does not constitute a convergence guarantee for algorithm~\ref{alg:gap_cont}, whose nonlinear subproblems are generally solved only locally and approximately.
More refined convergence analyses of complementarity regularization methods under additional constraint qualifications and second-order conditions can be found in, \eg, Scholtes~\cite{scholtes2001regularization} and Ralph and Wright~\cite{ralph2004regularization}.

\section{Implementation}\label{sec:impl}
We implement the proposed DBLP framework in the Python package BLVPY serving as a CVXPY~\cite{diamond2016cvxpy,agrawal2018rewriting,zhang2026cvxpy} extension for bilevel optimization, which is fully open-source and available at
\begin{quote}
    \url{https://github.com/dxogrp/blvpy}.
\end{quote}

\subsection{The BLVPY pipeline}
The solution pipeline of BLVPY can be roughly divided into four stages: audited canonicalization of the lower problem, construction of an initial lifted point, gap continuation through a DNLP solver, and independent verification of the returned numerical point.
The current release supports DBLP-compliant optimistic bilevel problems whose lower problems admit lossless canonical conic representations over zero, nonnegative, second-order, exponential, and 3-dimensional power cones.

\paragraph{Audited atoms and canonicalization.}
BLVPY replaces each linked upper CVXPY variable by a CVXPY parameter in a copy of the lower problem and verifies that the resulting problem satisfies the DCP and DPP rules.
DCP and DPP compliance alone, however, does not guarantee that CVXPY canonicalizes every modeled function without approximation or uses only cone blocks and reduction steps supported by BLVPY.
It therefore audits both the source expression tree and the CVXPY reduction chain before extracting the affine maps from the upper variables to $A(x)$, $b(x)$, $c(x)$, and $d(x)$, as well as the affine recovery map $\rho$.
Table~\ref{tab:supported_lower_atoms} lists the lower-level atoms audited by the current BLVPY release.
\begin{table}
    \centering
    \footnotesize
    \setlength{\tabcolsep}{3pt}
    \renewcommand{\arraystretch}{1.5}
    \caption{
        Lower-level atoms audited by the current BLVPY release.
        The last column summarizes the representation produced during canonicalization.
        LP denotes a linear program representation (\ie, involving only zero and nonnegative cones).
        SOC, EXP, and P3D denote second-order, exponential, and 3-dimensional power cones, respectively.
        We denote the $i$th-largest entry of a vector $z$ by $z_{[i]}$.
        }\label{tab:supported_lower_atoms}
    \begin{tabular}{@{}p{0.34\textwidth}p{0.42\textwidth}p{0.18\textwidth}@{}}
        \toprule
        \sffamily\bfseries Atom & \sffamily\bfseries Definition & \sffamily\bfseries Representation \\
        \midrule
        \multicolumn{3}{@{}l}{\itshape Affine expressions} \\
        \texttt{affine expression} & $\phi(z) = M z + q$ & Affine \\
        \midrule
        \multicolumn{3}{@{}l}{\itshape Elementwise atoms} \\
        \texttt{abs(z)} & ${\phi(z)}_i=|z_i|$ & LP \\
        \texttt{entr(z)} & ${\phi(z)}_i=-z_i\log z_i$ & EXP \\
        \texttt{exp(z)} & ${\phi(z)}_i=e^{z_i}$ & EXP \\
        \texttt{huber(z,M)} & ${\phi(z)}_i=\left\{
            \begin{array}{ll}
                z_i^2,&|z_i|\leq M\\
                2M|z_i|-M^2,&|z_i|>M
            \end{array}\right.$ & SOC \\[10pt]
        \texttt{kl\_div(z,v)} & ${\phi(z,v)}_i=z_i\log(z_i/v_i)-z_i+v_i$ & EXP \\
        \texttt{log(z)} & ${\phi(z)}_i=\log z_i$ & EXP \\
        \texttt{log1p(z)} & ${\phi(z)}_i=\log(1+z_i)$ & EXP \\
        \texttt{logistic(z)} & ${\phi(z)}_i=\log(1+e^{z_i})$ & EXP \\
        \texttt{maximum(z,v)} & ${\phi(z,v)}_i=\max\{z_i,v_i\}$ & LP \\
        \texttt{minimum(z,v)} & ${\phi(z,v)}_i=\min\{z_i,v_i\}$ & LP \\
        \texttt{power(z,p)} & ${\phi(z)}_i=z_i^p$ & Affine/SOC/P3D \\
        \texttt{rel\_entr(z,v)} & ${\phi(z,v)}_i=z_i\log(z_i/v_i)$ & EXP \\
        \texttt{xexp(z)} & ${\phi(z)}_i=z_i e^{z_i}$ & EXP/SOC \\
        \midrule
        \multicolumn{3}{@{}l}{\itshape Vector atoms} \\
        \texttt{cummax(z)} & ${\phi(z)}_i=\max\{z_1,\ldots,z_i\}$ & LP \\
        \texttt{dotsort(z,W)} & $\phi(z,W)=\sum_i z_{[i]}W_{[i]}$ & LP \\
        \texttt{geo\_mean(z,p)} & $\phi(z)=\prod_i z_i^{w_i}$, $w_i=p_i/\sum_j p_j$ & SOC \\
        \texttt{log\_sum\_exp(z)} & $\phi(z)=\log\sum_i e^{z_i}$ & EXP \\
        \texttt{max(z)} & $\phi(z) = \max_i z_i$ & LP \\
        \texttt{min(z)} & $\phi(z) = \min_i z_i$ & LP \\
        \texttt{norm1(z)} & $\phi(z)=\sum_i |z_i|$ & LP \\
        \texttt{norm\_inf(z)} & $\phi(z)=\max_i |z_i|$ & LP \\
        \texttt{pnorm(z,p)} & $\phi(z) = \norm{z}_p$ & LP/SOC/P3D \\
        \texttt{sum\_largest(z,k)} & $\phi(z)=\sum_{i=1}^k z_{[i]}$ & LP \\
        \midrule
        \multicolumn{3}{@{}l}{\itshape Quadratic atoms} \\
        \texttt{quad\_form(z,P)} & $\phi(z) = z^T P z$ & SOC \\
        \texttt{quad\_over\_lin(z,t)} & $\phi(z) = \norm{z}_2^2/t$ & SOC \\
        \bottomrule
    \end{tabular}
\end{table}

\paragraph{Initialization.}
BLVPY preserves a user-specified value for each upper variable.
Otherwise, letting $L$ and $U$ denote the possibly infinite lower and upper bounds on, say, the scalar variable $x$, respectively, it sets
\begin{equation*}
    x^{(0)} = \left\{
    \begin{array}{ll}
        (L+U)/2, & L, U \in \reals\\
        L + 1, & L \in \reals,\quad U = \infty\\
        U - 1, & L = -\infty,\quad U \in \reals\\
        0, & L = -\infty,\quad U = \infty.
    \end{array}\right.
\end{equation*}
Given this initialization, BLVPY first enforces the CVXPY attributes of the upper variables.
If the upper constraints, together with any domain restrictions inherited from the parameterized lower problem, form a DCP set, BLVPY then attempts a least distance projection onto this set.
After fixing the upper point, BLVPY solves the canonical lower conic problem to initialize $u$, $s$, and $\lambda$, and applies the solution recovery map $\rho$ to initialize the original lower variables.
If the resulting point fails the independent residual check, an optional restoration problem minimizes a common nonnegative radius that relaxes the lifted constraints and the complementarity gap constraint.
The point is used for continuation only if the recomputed violations satisfy the prescribed feasibility tolerance.
This initialization strategy turns out to work quite well in practice, but it is not guaranteed to succeed for all problems.
When this happens, the user may need to provide a feasible initial point based on problem specific knowledge.

\paragraph{Continuation and local search.}
BLVPY assembles the lifted DNLP problem once and updates only the parameter $\epsilon$, using every accepted point to warm-start the next nonlinear solve.
The default schedule starts at $\epsilon^{(0)}=10^{-1}$, contracts it by $0.1$, and terminates at $10^{-6}$; the default feasibility tolerance is $10^{-7}$.
When a scheduled reduction from the last accepted value $\epsilon$ to a target value $\tilde{\epsilon}$ fails, BLVPY inserts the intermediate value
\begin{equation*}
    \epsilon_{\rm mid}=\sqrt{\epsilon \tilde{\epsilon}}
\end{equation*}
and retries, with at most eight insertions between successive scheduled targets by default.
BLVPY also allows a multistart procedure as a heuristic for finding better local solutions.
In this setup, BLVPY runs $N$ independent continuation paths from eligible randomized upper initializations and returns, among the points accepted at the prescribed final tolerance, the one with the smallest upper objective.

\paragraph{Residual verification and output.}
BLVPY does not accept a point solely on the basis of the status reported by the nonlinear solver.
After every nonlinear solve, it independently recomputes the primal- and dual-equality residuals, recovery-map residual, upper-constraint violation, distances to the primal and dual cones, complementarity product, and relaxed-gap violation.
An attempted point is accepted only if the solver reports a solution and the recomputed violations satisfy the requested tolerance.
The returned result contains immutable snapshots of the source and canonical variables, the accepted and attempted $\epsilon$ histories, objectives, solver statuses, residuals, and all multistart runs.
An optional diagnostic solves the lower-level conic problem once more with the upper variable fixed at its returned value and reports the difference between the returned lower objective and the resulting optimal value.
IPOPT~\cite{wachter2006implementation} is the tested default DNLP backend, while Clarabel~\cite{goulart2024clarabel} is the default conic backend for initialization, projection, and diagnostics.

\subsection{A simple code snippet}
As a basic coding example, the following snippet shows how to specify and solve the bilevel problem defined in (\ref{prob:example_bilevel}) and (\ref{prob:example_lower}) using BLVPY\@.
More complex practical examples are presented later in \S\ref{sec:example}.
\begin{lstlisting}[language=zhpython]
import cvxpy as cp
import blvpy as bp  # import the BLVPY package

# define problem variables
x = cp.Variable(n)
y = cp.Variable(n)

# specify the lower problem
lower = bp.LowerProblem(
    cp.Minimize(cp.norm1(y - x)),  # lower objective
    [y >= 0],  # lower constraints
    parameters=[x],  # lower parameters
)

# specify the upper problem
obj = cp.Minimize(
    cp.sum_squares(x - g)
    + gamma * cp.sum_squares(y - h)
)  # upper objective
constr = [-r <= x, x <= r]  # upper constraints
prob = bp.BilevelProblem(obj, lower, constr)

# solve the bilevel problem
result = prob.solve()
\end{lstlisting}
Note that the \verb+BilevelProblem+ and \verb+LowerProblem+ classes share the same CVXPY variables \verb+x+ and \verb+y+, while the argument \verb+parameters=[x]+ tells BLVPY to treat the upper variable $x$ as fixed data in the lower problem.
The call to \verb+prob.solve()+ performs validation, canonicalization, and gap continuation.
The default solve uses a single deterministic initialization.
To request $N$ independent continuation runs from randomized upper-level initializations, the user can set, for example, \verb+x.sample_bounds = (-r, r)+ and pass \verb+best_of=N+ to the \verb+solve()+ method.

\section{Examples}\label{sec:example}
In this section, we present several examples of bilevel optimization problems from different application domains.
These examples are simplified and are by no means exhaustive.
Our goal here is to illustrate the basic modeling capabilities of the proposed DBLP framework and the BLVPY package.
Interested readers may refer to
\begin{quote}
    \url{https://github.com/dxogrp/blvpy/tree/main/examples}
\end{quote}
for more examples not covered in this paper.
All presented example bilevel problems were solved with the target complementarity gap $\epsilon=10^{-9}$, and all returned final solutions passed the validation checks and diagnostics described in \S\ref{sec:impl}; see also the example repository for full code scripts, synthetic dataset setups, and detailed numerical results.

\subsection{Ridge regression auto-tuning}
We consider the problem of automatically tuning the regularization parameter in \emph{ridge regression}~\cite{hoerl1970ridge,barratt2021least}.
Let $X_{\rm tr} \in \reals^{m_{\rm tr} \times n}$ and $y_{\rm tr} \in \reals^{m_{\rm tr}}$ denote the training data, and let $X_{\rm val} \in \reals^{m_{\rm val} \times n}$ and $y_{\rm val} \in \reals^{m_{\rm val}}$ denote the validation data.
The hyperparameter tuning problem of ridge regression can be formulated as
\begin{equation}\label{prob:ridge}
    \begin{array}{ll}
        \mbox{minimize} & (1/m_{\rm val})\norm{X_{\rm val} w - y_{\rm val}}_2^2\\
        \mbox{subject to} & 10^{-4} \leq \lambda \leq 10\\
        & w \in S(\lambda),
    \end{array}
\end{equation}
where
\begin{equation}\label{prob:ridge_lower}
    \begin{array}{rl}
        S(\lambda) = \argmin_w & (1/m_{\rm tr})\norm{X_{\rm tr} w - y_{\rm tr}}_2^2 + \lambda \norm{w}_2^2.
    \end{array}
\end{equation}
The bilevel problem (\ref{prob:ridge}) has upper variable $\lambda \in \reals_+$ and lower variable $w \in \reals^n$, which correspond to the regularization parameter and the regression coefficients, respectively.
We can interpret the lower problem (\ref{prob:ridge_lower}) as a ridge regression problem with a fixed regularization parameter $\lambda$, and the upper problem (\ref{prob:ridge}) as a (constrained) least squares problem that tunes $\lambda$ to minimize the validation error.

To specify the bilevel problem (\ref{prob:ridge}) in BLVPY, we can use the following code snippet:
\begin{lstlisting}[language=zhpython]
# problem variables
lbd = cp.Variable(nonneg=True)
w = cp.Variable(n)

# lower problem
tr_loss = cp.sum_squares(X_tr @ w - y_tr) / m_tr
lower = bp.LowerProblem(
    cp.Minimize(tr_loss + lbd * cp.sum_squares(w)),
    parameters=[lbd],
)

# upper problem
val_loss = cp.sum_squares(X_val @ w - y_val) / m_val
prob = bp.BilevelProblem(
    cp.Minimize(val_loss),
    lower,
    [lbd >= 1e-4, lbd <= 10],
)
\end{lstlisting}

We solve the problem (\ref{prob:ridge}) with BLVPY on a randomly generated dataset.
Figure~\ref{fig:ridge} shows the training and validation losses as functions of the regularization parameter $\lambda$.
\begin{figure}
    \centering
    \includegraphics[width=0.85\textwidth]{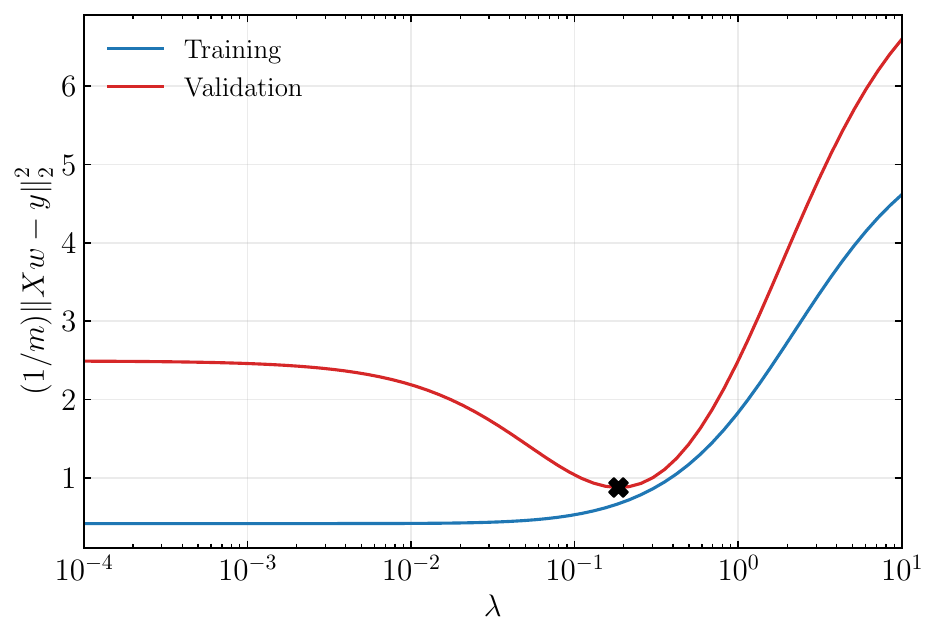}
    \caption{
        \emph{Ridge regression auto-tuning.}
        Training and validation losses as functions of the regularization parameter $\lambda$.
        The black cross indicates the final solution of the problem (\ref{prob:ridge}) returned by BLVPY\@.
        Note that the subscripts are removed from the axes labels for simplicity.
    }\label{fig:ridge}
\end{figure}
The figure shows that the final solution returned by BLVPY (shown as a black cross) corresponds to a regularization parameter that approximately achieves the minimum validation loss.

\subsection{Stackelberg security game}
We consider a Stackelberg security game in which a defender allocates limited security resources to protect a set of targets against an attacker~\cite{shieh2012protect,guo2019inducibility}.
Specifically, suppose there are $n$ targets, and let $p \in \reals^n$ with $0 \preceq p \preceq \ones$ denote the probability that the defender protects each target, and let $q \in \reals^n$ with $0 \preceq q \preceq \ones$ denote the probability that the attacker attacks each target.
For a fixed defender allocation $p$, the attacker chooses a strategy $q$ to solve the lower problem
\begin{equation}\label{prob:stackelberg_lower}
    \begin{array}{rl}
        S(p) = \argmin_q & - \sum_{i = 1}^n (r_i - c_i p_i)q_i\\
        \mbox{subject to} & q \succeq 0,\quad \ones^T q = 1,
    \end{array}
\end{equation}
where $r, c \in \reals^n$ are given attacker reward and defender coverage effectiveness vectors, respectively, so that $r_i - c_i p_i$ is the attacker's payoff for the $i$th target after observing defender coverage.
The defender then adapts its protection strategy to solve the upper problem
\begin{equation}\label{prob:stackelberg}
    \begin{array}{ll}
        \mbox{minimize} & \sum_{i = 1}^n L_i(1 - p_i) q_i + \gamma \norm{p}_2^2\\
        \mbox{subject to} & 0 \preceq p \preceq \ones,\quad \ones^T p \leq B\\
        & q \in S(p),
    \end{array}
\end{equation}
where $p, q$ are the upper and lower variables, respectively, $L \in \reals^n$ is a given (uncovered) defender loss vector under attack, $\gamma > 0$ is a given penalty parameter, and $B > 0$ is a given budget for the defender.
Note that the lower program (\ref{prob:stackelberg_lower}) can have multiple optimal attack targets and the optimistic bilevel semantics considered here is consistent with the strong-Stackelberg convention, \ie, if $S(p)$ contains multiple attacker best responses, the defender may select the one that minimizes its upper-level objective in (\ref{prob:stackelberg}).

To specify the problem (\ref{prob:stackelberg}) in BLVPY, we can use the following code snippet:
\begin{lstlisting}[language=zhpython]
# problem variables
p = cp.Variable(n, nonneg=True)
q = cp.Variable(n, nonneg=True)

# lower problem
utility = r - cp.multiply(c, p)
lower = bp.LowerProblem(
    cp.Minimize(-utility @ q),
    [cp.sum(q) == 1],
    parameters=[p],
)

# upper problem
loss = cp.multiply(L, 1.0 - p) @ q
prob = bp.BilevelProblem(
    cp.Minimize(loss + gamma * cp.sum_squares(p)),
    lower,
    [p <= 1, cp.sum(p) <= B],
)
\end{lstlisting}

Figure~\ref{fig:stackelberg} shows the optimized defender allocation and the corresponding attacker response returned by BLVPY for a randomly generated instance of the problem (\ref{prob:stackelberg}).
\begin{figure}
    \centering
    \includegraphics[width=\textwidth]{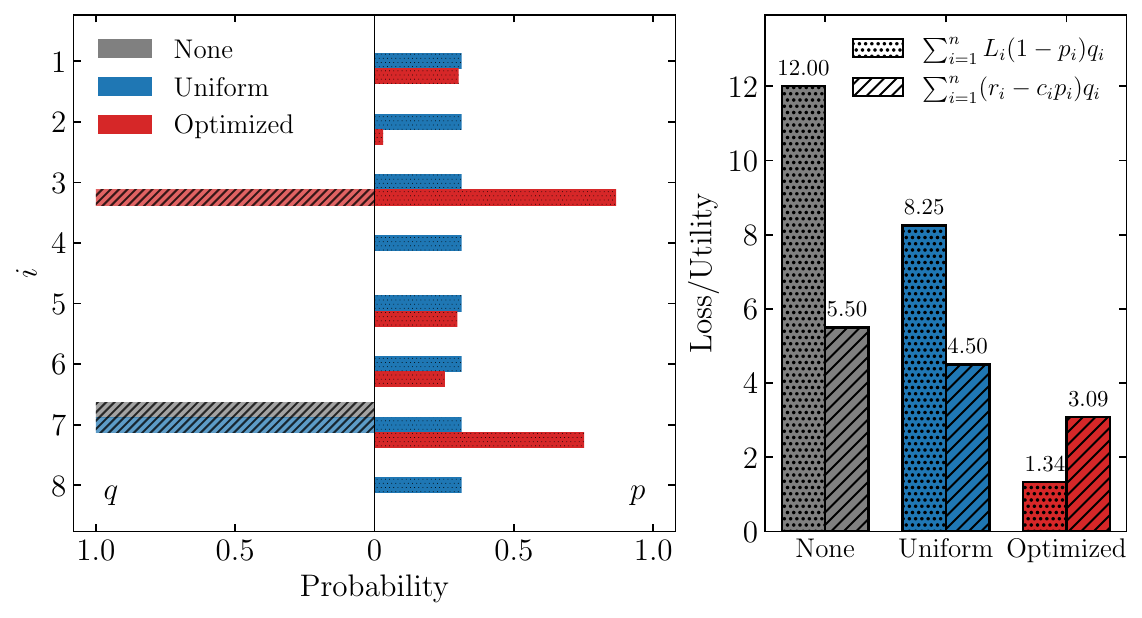}
    \caption{
        \emph{Stackelberg security game.}
        \emph{Left.}
        Defender protection $p$ and attacker response $q$ under none, uniform, and optimized protection.
        \emph{Right.}
        The corresponding expected defender loss and attacker utility.
        All numerical quantities are reported in normalized units.
    }\label{fig:stackelberg}
\end{figure}
For reference, we also plot the allocations with no and uniform defender coverage and the corresponding attacker responses.
The optimized defender allocation returned by BLVPY reduces the expected defender loss by approximately $80\%$ compared with uniform coverage and by approximately $90\%$ compared with no coverage.

\subsection{Renewable capacity planning}
We consider the problem of renewable capacity planning in a power system~\cite{zeng2017bi}, where a generation planner must choose renewable capacity before an electricity dispatcher sees the resulting operating limits and optimizes the dispatch accordingly.
Let $\alpha \in \reals_+$ denote the renewable capacity, and $L \in \reals^m_+$ denote the renewable generation availability at $m$ different time periods, so the vector $\alpha L \in \reals^m_+$ represents the renewable generation limits at all time periods.
Let $D \in \reals^m_+$ denote the electricity demand at all time periods, and $r, g \in \reals^m_+$ denote the renewable and thermal generation dispatch at all time periods.
A generation planner may formulate the renewable capacity planning problem as
\begin{equation}\label{prob:grid}
    \begin{array}{ll}
        \mbox{minimize} & p_r \alpha^2 + p_g \ones^T g\\
        \mbox{subject to} & 0 \leq \alpha \leq \alpha_{\rm max}\\
        & (r,g) \in S(\alpha),
    \end{array}
\end{equation}
where a grid dispatcher solves the lower problem
\begin{equation*}
    \begin{array}{rl}
        S(\alpha) = \argmin_{r,g} & c_r \ones^T r + c_g \ones^T g + \delta(\norm{r}_2^2 + \norm{g}_2^2)\\
        \mbox{subject to} & 0 \preceq r \preceq \alpha L,\quad g \succeq 0\\
        & r + g \succeq D.
    \end{array}
\end{equation*}
The problem (\ref{prob:grid}) has upper variable $\alpha$ and lower variables $r, g$.
The planning and dispatching cost coefficients $p_r, p_g, c_r, c_g$, the maximum renewable capacity $\alpha_{\rm max}$, the regularization parameter $\delta > 0$, as well as the vectors $L, D$ are assumed to be given.

To specify the bilevel problem (\ref{prob:grid}) in BLVPY, we can use the following code snippet:
\begin{lstlisting}[language=zhpython]
# problem variables
alpha = cp.Variable()
r = cp.Variable(m)
g = cp.Variable(m)

# lower problem
lower = bp.LowerProblem(
    cp.Minimize(
        c_r * cp.sum(r) + c_g * cp.sum(g) 
        + delta * (cp.sum_squares(r) + cp.sum_squares(g))),
    [
        r >= 0.0,
        r <= cp.multiply(L, alpha),
        g >= 0.0,
        r + g >= D,
    ],
    parameters=[alpha],
)

# upper problem
prob = bp.BilevelProblem(
    cp.Minimize(p_r * cp.square(alpha) + p_g * cp.sum(g)),
    lower,
    [alpha >= 0.0, alpha <= alpha_max],
)
\end{lstlisting}

We solve the problem (\ref{prob:grid}) with BLVPY on a synthetic dataset, and the final solution returned by BLVPY is shown in figure~\ref{fig:grid}.
\begin{figure}
    \centering
    \includegraphics[width=0.85\textwidth]{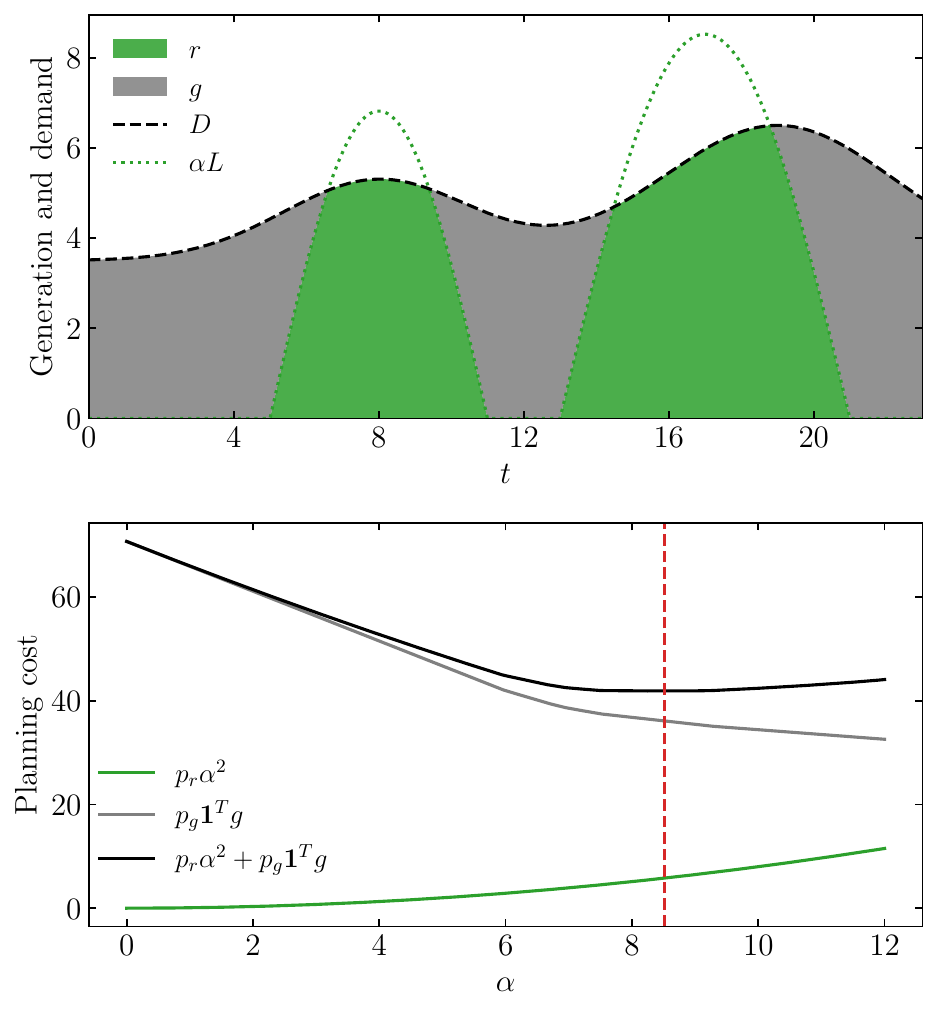}
    \caption{
        \emph{Renewable capacity planning.}
        \emph{Top.}
        Optimal renewable $r$ and thermal $g$ dispatch under demand $D$ and the optimal renewable limit $\alpha L$ obtained from the BLVPY solution.
        \emph{Bottom.}
        Investment and thermal generation costs versus capacity, with the dashed line marking the final solution of the problem (\ref{prob:grid}) returned by BLVPY\@.
    }\label{fig:grid}
\end{figure}
The lower plot indicates that the BLVPY solution (shown as a red dashed line) approximately minimizes the total renewable and thermal planning cost, while the upper plot shows the corresponding optimal dispatch under the given demand and renewable limit.

\subsection{Traffic tolling}
We consider second-best congestion pricing on a Braess network~\cite{braess2005paradox,verhoef2002second}.
Specifically, consider a network with four nodes $O, A, B, D$, in which travelers move from the origin $O$ to the destination $D$.
As illustrated in figure~\ref{fig:network}, the ordered edge and path lists are
\begin{equation*}
    E = (OA, OB, AD, BD, AB)\quad\mbox{and}\quad P = (OAD, OBD, OABD),
\end{equation*}
respectively.
The authority may impose a nonnegative toll $\tau \in \reals_+$ on the central edge $AB$, which is used only by the third path.
\begin{figure}
    \centering
    \includegraphics[width=0.45\textwidth]{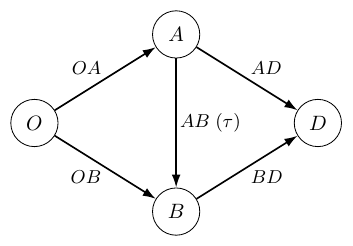}
    \caption{
        \emph{Braess network used in the traffic-tolling example.}
        The arrows indicate the permitted travel directions, and $\tau$ denotes the toll imposed on the central edge $AB$.
    }\label{fig:network}
\end{figure}
Let $x \in \reals^3$ collect the path flows in the order specified by $P$, and let
$f = (f_{OA}, f_{OB}, f_{AD}, f_{BD}, f_{AB}) \in \reals^5$ collect the edge flows in the order specified by $E$.
The path-edge incidence relation is given by
\begin{equation*}
    f = Hx,\qquad H = \left[\begin{array}{ccc}
        1 & 0 & 1\\
        0 & 1 & 0\\
        1 & 0 & 0\\
        0 & 1 & 1\\
        0 & 0 & 1
    \end{array}\right],
\end{equation*}
where $H_{ei} = 1$ if edge $e$ belongs to path $i$, and $H_{ei} = 0$ otherwise.
In particular, the last row gives $f_{AB} = x_3$.
For a total travel demand $q \in \reals_+$, feasible path flows satisfy $x \succeq 0$ and $\ones^T x = q$, which also implies $f = Hx \succeq 0$.

Let $e \in E$ denote an edge in the network, and let $f_e$ denote the flow on that edge.
We assume that the travel time on edge $e$ is an affine function of the flow $f_e$, given by
\begin{equation*}
    T_e(f_e) = b_e + a_e f_e,
\end{equation*}
where $a_e, b_e \geq 0$ are given edge-specific parameters.
For a toll $\tau$ imposed on the central edge $AB$, we assume that travelers choose their paths according to the Wardrop equilibrium~\cite{wardrop1952road}, which can be formulated as the following lower problem:
\begin{equation*}
    \begin{array}{rl}
        S(\tau) = \argmin_x & \sum_{e \in E} (b_e f_e + (1/2) a_e f_e^2) + \tau f_{AB}\\
        \mbox{subject to} & x \succeq 0,\quad \ones^T x = q\\
        & f = Hx,
    \end{array}
\end{equation*}
which is sometimes called the Beckmann potential formulation~\cite[\S3.1.2]{beckmann1956studies} of the Wardrop equilibrium.
The authority's objective is to minimize the total travel time of all travelers plus a penalty term for the toll, which can be expressed as the upper problem
\begin{equation}\label{prob:traffic}
    \begin{array}{ll}
        \mbox{minimize} & \sum_{e \in E} f_e T_e(f_e) + \rho \tau^2\\
        \mbox{subject to} & 0 \leq \tau \leq \tau_{\rm max}\\
        & x \in S(\tau),
    \end{array}
\end{equation}
where $x$ is the lower variable, $\tau$ is the upper variable, and $\rho > 0$ is a given penalty parameter.

To specify the problem (\ref{prob:traffic}) in BLVPY, we can use the following code snippet:
\begin{lstlisting}[language=zhpython]
# problem variables
tau = cp.Variable(1)
x = cp.Variable(3)
f = H @ x

# lower problem
lower = bp.LowerProblem(
    cp.Minimize(0.5 * a @ cp.square(f) + b @ f + tau[0] * x[2]),
    [x >= 0, cp.sum(x) == q],
    parameters=[tau],
)

# upper problem
time = b @ f + a @ cp.square(f)
prob = bp.BilevelProblem(
    cp.Minimize(time + rho * cp.sum_squares(tau)),
    lower,
    [tau >= 0, tau <= tau_max],
)
\end{lstlisting}

We solve the problem (\ref{prob:traffic}) with BLVPY on a synthetic dataset.
The two network diagrams on the left of figure~\ref{fig:traffic} compare the edge flows under the optimal toll $\tau^\star = 2.03$ returned by BLVPY for the problem (\ref{prob:traffic}), with those under no toll (\ie, $\tau = 0$).
\begin{figure}
    \centering
    \includegraphics[width=\textwidth]{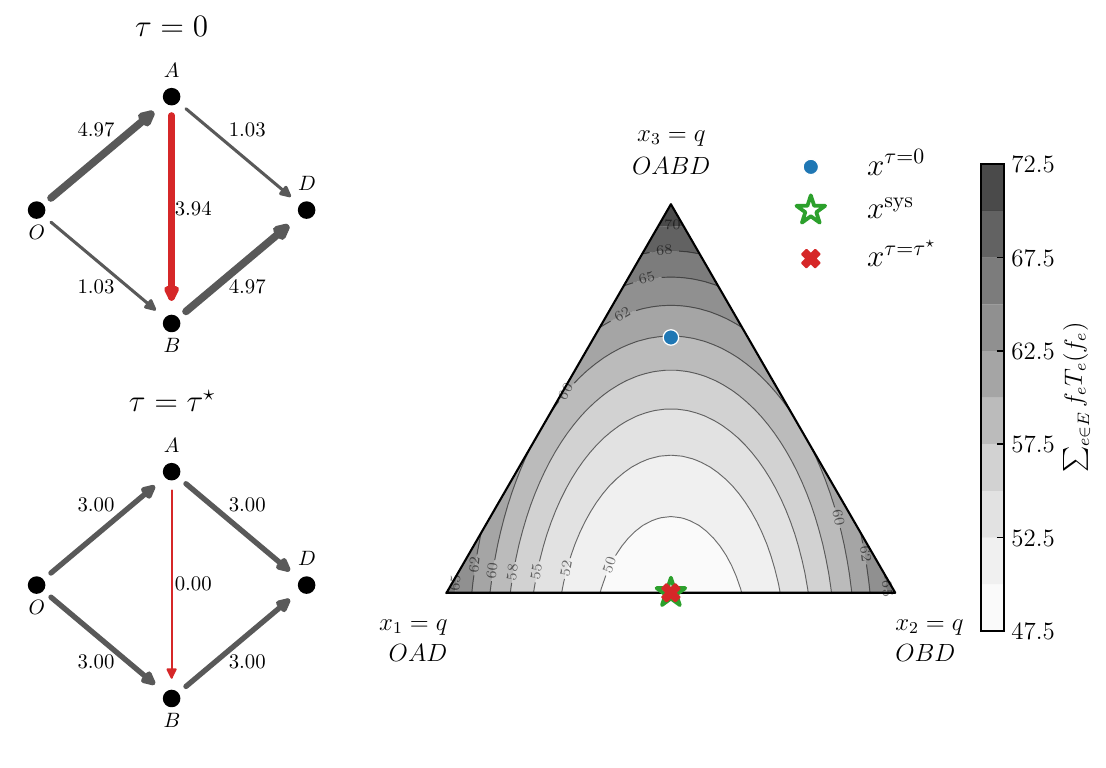}
    \caption{
        \emph{Traffic tolling on a Braess network.}
        \emph{Left.}
        Edge flows in untolled ($\tau = 0$) and optimally tolled ($\tau = \tau^\star$) Braess networks returned by BLVPY\@.
        \emph{Right.}
        Total travel time over feasible path flow allocations, where $x^{\rm sys}$ denotes the system optimal flow obtained by directly minimizing the total travel time assuming a central controller could assign routes to travelers.
        All numerical quantities are reported in normalized units.
    }\label{fig:traffic}
\end{figure}
With $\tau = 0$, approximately $3.94$ flow units use the central path $OABD$, which therefore raises network travel time to $59.89$ units.
The selected toll $\tau^\star = 2.03$ removes almost all flow from $AB$ and divides
demand between $OAD$ and $OBD$.
As a result, network travel time falls by approximately $20\%$ to $48.18$ units.
This is essentially the system optimal flow, which is achieved by directly minimizing the total travel time (\ie, the objective of the upper problem (\ref{prob:traffic})), as if a central controller could directly assign routes to travelers.
The right panel in figure~\ref{fig:traffic} shows total network travel time over every feasible path flow allocation satisfying $x_1 + x_2 + x_3 = q$.
For the untolled flow $x^{\tau = 0}$ (shown as a blue dot), the Wardrop equilibrium is not system optimal, while the optimal toll flow $x^{\tau = \tau^\star}$ returned by BLVPY (shown as a red cross) is very close to the system optimal flow $x^{\rm sys}$ (shown as a green star).

\subsection{Planar truss design}
We consider the problem of designing a planar truss structure (see, \eg, \cite[chapter 5]{christensen2009structural} and \cite[chapters 2 and 11]{terlaky2017advances}).
Let $a \in \reals^n_+$ and $L \in \reals^n_+$ denote the cross-sectional areas and lengths of the $n$ truss members, respectively, and let $E \in \reals$ denote the Young's modulus of the truss material.
Let $u \in \reals^m$ denote the nodal displacements for a given set of applied nodal forces $f \in \reals^m$, where $m$ represents the number of displacement degrees of freedom.
The designer's upper problem is to minimize the compliance $f^T u$ of the truss structure, subject to a budget constraint on the total material volume, \ie,
\begin{equation}\label{prob:truss}
    \begin{array}{ll}
        \mbox{minimize} & f^T u\\
        \mbox{subject to} & a_{\rm min} \ones \preceq a \preceq a_{\rm max} \ones\\
        & L^T a \leq V_{\rm max}\\
        & u \in S(a),
    \end{array}
\end{equation}
where $a_{\rm min}, a_{\rm max} > 0$ are given lower and upper bounds on the cross-sectional areas, and $V_{\rm max} > 0$ is a given maximum material volume.
We assume that the length vector $L \succ 0$ is given and fixed, so the problem (\ref{prob:truss}) has upper variable $a$ and lower variable $u$.
The lower problem is the linear elasticity problem that computes the nodal displacements $u$ for a given cross-sectional area allocation $a$, which can be expressed as
\begin{equation*}
    \begin{array}{rl}
        S(a) = \argmin_u & (1/2) u^T K(a) u - f^T u\\
        \mbox{subject to} & u_i = 0,\quad i \in \fset,
    \end{array}
\end{equation*}
where $\fset$ is the set of fixed displacement indices, and
\begin{equation*}
    K(a) = B^T \diag\left(\frac{E a_1}{L_1}, \ldots, \frac{E a_n}{L_n}\right) B
\end{equation*}
is the stiffness matrix of the truss structure, where $B \in \reals^{n \times m}$ is the compatibility matrix that maps nodal displacements to member elongations.

Suppose a structural designer needs to allocate material among the members of a two-bay cantilever truss, as shown in figure~\ref{fig:truss}.
\begin{figure}
    \centering
    \includegraphics[width=0.5\textwidth]{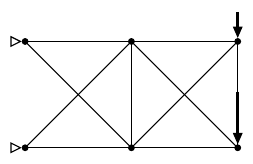}
    \caption{
        \emph{Two-bay cantilever truss.}
        The circles and solid lines represent the nodes and members of the truss, respectively.
        The arrows indicate the applied nodal forces, and the leftmost triangles indicate the supported nodes with fixed displacements.
    }\label{fig:truss}
\end{figure}
The truss has $6$ nodes and $n = 10$ members, and the number of displacement degrees of freedom is $m = 12$ (since each node has $2$ displacement degrees of freedom for the horizontal and vertical directions).
Both degrees of freedom at each of the two leftmost nodes are fixed, and a unit downward load is applied at the lower-right node and a smaller downward load at the upper-right node (shown by the arrows).
The following code snippet shows how to specify the corresponding bilevel problem (\ref{prob:truss}) in BLVPY for this example.
\begin{lstlisting}[language=zhpython]
# problem variables
a = cp.Variable(n, nonneg=True)
u = cp.Variable(m)

# lower problem
energy = 0.5 * cp.sum(cp.multiply(
    cp.multiply(E / L, a), cp.square(B @ u)
))
lower = bp.LowerProblem(
    cp.Minimize(energy - f @ u),
    [u[F] == 0],
    parameters=[a],
)

# upper problem
problem = bp.BilevelProblem(
    cp.Minimize(f @ u),
    lower,
    [L @ a <= V_max, a_min <= a, a <= a_max],
)
\end{lstlisting}

Figure~\ref{fig:truss_design} shows the optimized cross-sectional areas returned by BLVPY for the problem (\ref{prob:truss}) with synthetic numerical problem data.
\begin{figure}
    \centering
    \includegraphics[width=\textwidth]{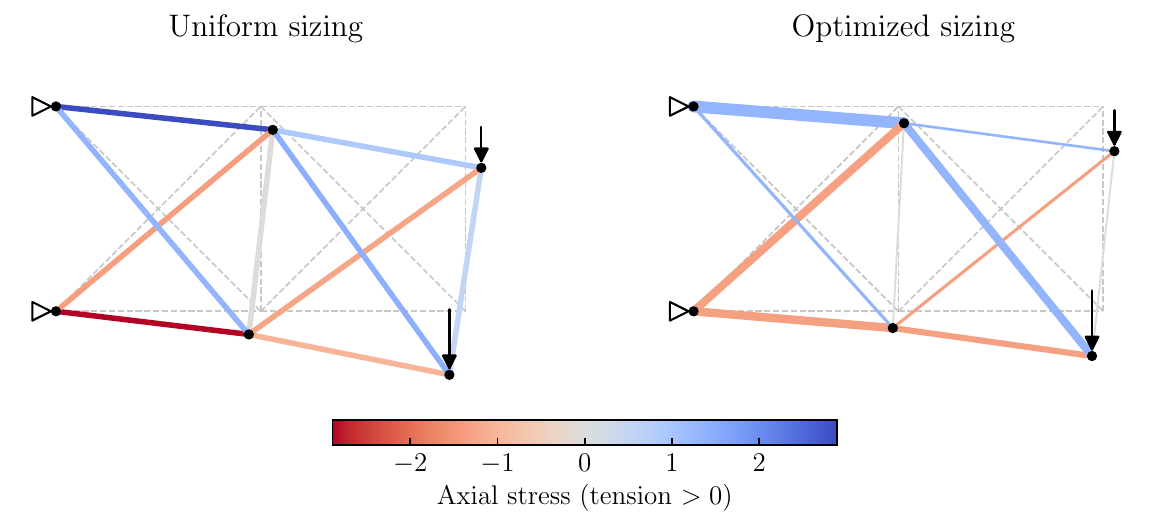}
    \caption{
        \emph{Planar truss design.}
        Uniform baseline (shown left) and optimized (shown right) cross-sectional area designs of the two-bay cantilever truss in figure~\ref{fig:truss}.
        The dashed lines indicate the undeformed geometry of the truss, while the solid lines indicate the deformed geometry under the applied loads.
        The color of each member indicates the magnitude of the axial stress in that member under equilibrium.
    }\label{fig:truss_design}
\end{figure}
A thicker line indicates a larger cross-sectional area, and the color of each member indicates the magnitude of the axial stress in that member under the optimized design.
We also show the uniform design with the same total material volume for comparison.
The figure shows that the optimized design returned by BLVPY allocates more material to members under higher stress.
The compliances of the uniform and optimized designs are $1.90$ and $1.35$ (normalized units), respectively, so the optimized design reduces compliance by approximately $30\%$ without using additional material.

\subsection{Differentiable model predictive control}
We consider the problem of learning model predictive control (MPC) regularization weights from expert demonstrations for a direct-current motor speed control task~\cite{amos2018differentiable,rawlings2026mpc}.
Specifically, the controller problem is a finite-horizon MPC problem that computes the optimal control actions for given regularization weights, and the learner problem is an imitation learning problem that tunes the regularization weights to minimize the discrepancy between the MPC output and expert demonstrations.
Conventionally, such a hierarchical MPC problem is solved by differentiating through the MPC controller problem (see, \eg, \cite{amos2018differentiable} and \cite{agrawal2021learning}).
Here we show that it can also be formulated as a bilevel problem and solved with BLVPY\@.

Let $x(t) = (\omega(t), i(t)) \in \reals^2$ denote the state of the motor at time $t = 0, 1, \ldots, N$, where $\omega(t)$ and $i(t)$ are the angular speed and current, respectively, and let $u(t) \in \reals$ denote the control input (voltage) at time $t = 0, 1, \ldots, N-1$.
We collect the complete state and control trajectories as $x = \left[\begin{array}{ccc}
    x(0) & \cdots & x(N)
\end{array}\right] \in \reals^{2 \times (N+1)}$ and $u = (u(0), \ldots, u(N-1)) \in \reals^N$, respectively, with the previous input $u(-1) = 0$ as fixed data.
The motor dynamics can be expressed as a linear discrete-time system
\begin{equation*}
    x(t+1) = A x(t) + B u(t),
\end{equation*}
where $A \in \reals^{2 \times 2}$ and $B \in \reals^{2 \times 1}$ are given system matrices.
We consider an MPC controller that solves the following lower problem for given regularization weights $\delta, \eta \in \reals_+$:
\begin{equation*}
    \begin{array}{rl}
        S(\delta, \eta) = \argmin_{x, u} & \sum_{t = 1}^{N} {(\omega(t) - 1)}^2 + 5 {(\omega(N) - 1)}^2 + 0.04 \sum_{t = 1}^{N} {(i(t) - 0.1)}^2\\[10pt]
        &\qquad + \delta \sum_{t = 0}^{N-1} {(u(t) - 0.2)}^2 + \eta \sum_{t = 0}^{N-1} {(u(t) - u(t-1))}^2\\[10pt]
        \mbox{subject to} & x(0) = (0, 0),\quad x(t+1) = A x(t) + B u(t),\quad t = 0, \ldots, N-1\\[5pt]
        & |u(t)| \leq 2,\quad t = 0, \ldots, N-1\\[5pt]
        & |i(t)| \leq 0.5,\quad -0.05 \leq \omega(t) \leq 1.25,\quad t = 0, \ldots, N.
    \end{array}
\end{equation*}
In this example, we learn the weights $\delta$ for the control magnitude and $\eta$ for the control rate of change from expert state and control trajectories $\hat{x}(t)$, $t = 0, \ldots, N$ and $\hat{u}(t)$, $t = 0, \ldots, N-1$ over the same horizon, while holding all other weights fixed.
The corresponding upper problem is given by
\begin{equation}\label{prob:mpc}
    \begin{array}{ll}
        \mbox{minimize} & \sum_{t = 0}^{N} \norm{x(t) - \hat{x}(t)}_2^2 + 0.1 \sum_{t = 0}^{N-1} \norm{u(t) - \hat{u}(t)}_2^2\\[5pt]
        \mbox{subject to} & 0.01 \leq \delta \leq 1,\quad 0.01 \leq \eta \leq 1\\
        & (x, u) \in S(\delta, \eta),
    \end{array}
\end{equation}
where $\delta, \eta$ are the upper variables and $x, u$ are the lower variables.
The following code snippet shows how to specify the bilevel problem (\ref{prob:mpc}) in BLVPY\@.
\begin{lstlisting}[language=zhpython]
# problem variables
delta = cp.Variable(nonneg=True)
eta = cp.Variable(nonneg=True)
x = cp.Variable((2, 25))
u = cp.Variable(24)

# lower problem
u_diff = cp.diff(cp.hstack([0, u]))
mpc_obj = (
    cp.sum_squares(x[0, 1:] - 1)
    + 5 * cp.square(x[0, -1] - 1)
    + 0.04 * cp.sum_squares(x[1, 1:] - 0.1)
    + delta * cp.sum_squares(u - 0.2)
    + eta * cp.sum_squares(u_diff)
)
x_next = A @ x[:, :-1] + cp.outer(B, u)
lower = bp.LowerProblem(
    cp.Minimize(mpc_obj),
    [
        x[:, 0] == np.zeros(2),
        x[:, 1:] == x_next,
        cp.abs(u) <= 2,
        cp.abs(x[1, :]) <= 0.5,
        x[0, :] >= -0.05,
        x[0, :] <= 1.25,
    ],
    parameters=[delta, eta],
)

# upper problem
loss = (
    cp.sum_squares(x - x_hat)
    + 0.1 * cp.sum_squares(u - u_hat)
)
problem = bp.BilevelProblem(
    cp.Minimize(loss),
    lower,
    [
        delta >= 0.01, delta <= 1.0,
        eta >= 0.01, eta <= 1.0,
    ]
)
\end{lstlisting}

\begin{figure}
    \centering
    \includegraphics[width=\textwidth]{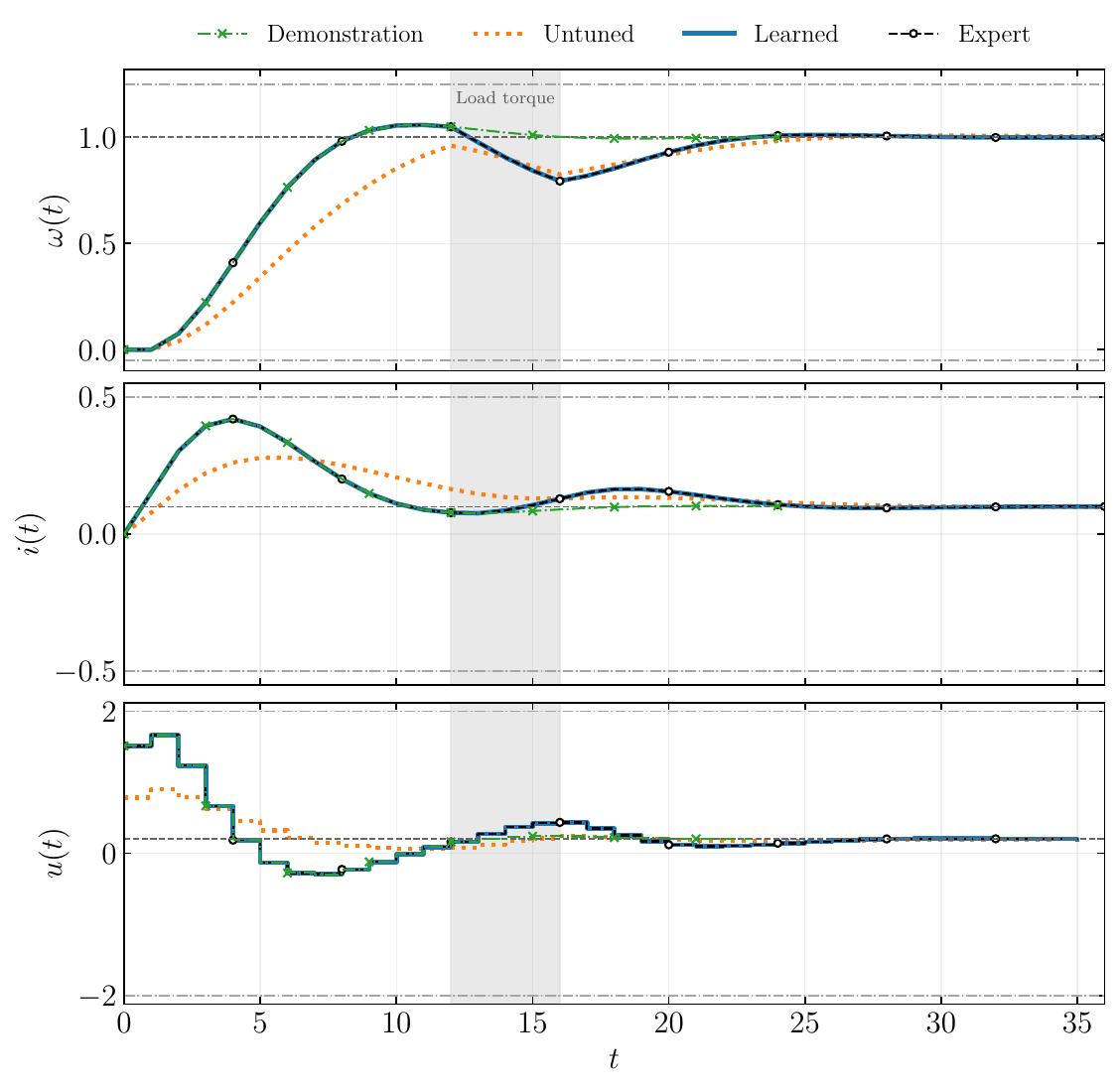}
    \caption{
        \emph{Differentiable model predictive control.}
        Validation trajectories for angular speed $\omega(t)$ (top), current $i(t)$ (middle), and control input $u(t)$ (bottom) under the untuned baseline controller with uniform weights $\delta = \eta = 0.5$ (shown in orange) and the learned controller with weights returned by BLVPY (shown in blue).
        The expert demonstration used for imitation learning and the expert validation trajectory are shown in green and black, respectively.
        The shaded region indicates the duration of an unannounced load torque pulse.
        The horizontal dashed and dashdotted lines plot the target values and corresponding safety limits, respectively.
    }\label{fig:mpc}
\end{figure}
We solve the problem (\ref{prob:mpc}) with BLVPY using synthetic expert demonstrations, then validate the learned weights in a 36-step receding horizon rollout with an unannounced load torque pulse.
We also compare the learned weights with an untuned baseline controller that uses uniform weights $\delta = \eta = 0.5$.
The resulting closed-loop trajectories are shown in figure~\ref{fig:mpc}.
The learned weights returned by BLVPY reproduce the expert behavior both in the nominal case and under the unannounced load torque pulse.
In particular, on the nominal demonstration trajectory, the imitation loss (\ie, the upper objective in problem (\ref{prob:mpc})) decreases from $0.90$ for the untuned baseline to essentially zero for the learned controller.

\section{Discussion}\label{sec:discussion}
\subsection{Conclusion}
In this paper, we introduce DBLP, a disciplined framework for modeling and solving bilevel optimization problems.
If a bilevel problem has a DNLP-compliant upper problem and a DPP-compliant convex lower problem, DBLP can then automatically constructs a single-level reformulation of the original problem via the conic KKT conditions and solves it using a gap continuation procedure.
Based on these ideas, we implement the open-source BLVPY package as an extension to CVXPY, which allows users to specify bilevel problems in a high-level, human-readable way, while automatically handling the low-level transformations and solver interfacing.

To illustrate the expressive scope of DBLP, we consider examples from several application domains.
In each case, the bilevel model can be specified compactly in a few lines of human-readable code close to the original mathematical formulation, while DBLP automatically constructs the single-level reformulation and interfaces it with the numerical solver.
By separating high-level model specification from these low-level transformations, DBLP is intended to make bilevel programming more accessible in practice.
We hope that this will facilitate its use in new application areas.

\subsection{Solution quality and guarantees}
We should emphasize that the DBLP framework and gap continuation procedure do \emph{not} guarantee global optimality of the computed solution, even when the upper-level objective and constraints are convex.
This is because the single-level reformulation (\ref{prob:single_level}) and its relaxation (\ref{prob:approx_single_level}) are generally nonconvex and are therefore solved only approximately by nonlinear solvers.
In the most general case, little can be said about the convergence properties of generic nonlinear numerical solvers; interested readers may refer to Cederberg~\etal~\cite[\S2.3]{cederberg2026disciplined} for a detailed discussion.

What DBLP \emph{does} provide is a lower-level certificate:
If the final point satisfies the constraints of (\ref{prob:approx_single_level}) for a tolerance $\epsilon$, proposition~\ref{prop:lower_certificate} ensures that the recovered lower point is feasible and at most $\epsilon$-suboptimal for the lower problem, \ie, roughly speaking, the returned point is approximately bilevel-feasible for the original problem (\ref{prob:blv}).
Our examples show that such a point is often a good (or at least acceptable) solution in practice.

\subsection{Lower-suboptimality and polishing}
It can happen in practice that for a given bilevel problem of the form (\ref{prob:blv}), the continuation procedure described in \S\ref{sec:gap_continuation} cannot reach a requested tolerance $\epsilon$.
If this point was accepted at a larger tolerance, say, $\tilde{\epsilon} > \epsilon$, proposition~\ref{prop:lower_certificate} still certifies lower-level feasibility and at most $\tilde{\epsilon}$ lower-level suboptimality, and the point may therefore be reported as an approximate solution of (\ref{prob:blv}) together with its actual complementarity gap and residuals.

On the other hand, as a post-processing step, one may fix the upper variable at its returned value and solve the associated convex lower problem (\ref{prob:blv_lower}) again to obtain a more accurate lower-level response.
Computationally, this polishing step is typically inexpensive and numerically reliable since it requires only a single convex solve.
It tends to work particularly well when the lower solution is unique and stable with respect to the upper variable and the upper constraints have sufficient slack, so that the change in the lower variable does not destroy upper-level feasibility.
We should note, however, that this replacement is not an automatic repair, since it can change the upper objective, violate upper constraints that depend on the lower variable, and, when the lower problem has multiple solutions, return a response that is not the one selected by the optimistic bilevel problem.
The polished upper-lower solution pair should therefore be reevaluated against all upper constraints and objectives; unless it passes these checks, the resolve should be treated as a diagnostic rather than as a certified solution of the original bilevel problem.
If the polished solution is indeed feasible and improves the upper objective, it may be accepted as a better solution of the original problem.

To support this post-solve polishing procedure, BLVPY provides a \texttt{polish} method on the \texttt{BilevelProblem} class that can be applied directly to any solution returned by \texttt{solve}, so that a user can perform this step directly in native BLVPY workflow without manually reconstructing and initializing the lower convex problem again in CVXPY.

\newpage
\appendix
\section{Proofs}
In this appendix, we provide the detailed proofs of the results stated in the main text.

\subsection{Proof of proposition~\ref*{prop:lower_certificate}}\label{sec:low_cert_proof}
\begin{proof}
Fix a feasible point $(x,u,s,\lambda)$ of (\ref{prob:approx_single_level}), and denote its canonical primal and dual objective values by
\begin{equation*}
    P = {c(x)}^T u + d(x)\quad\mbox{and}\quad
    D = -{b(x)}^T \lambda + d(x),
\end{equation*}
respectively.
Primal feasibility gives $b(x) = A(x)u + s$, while dual feasibility gives ${A(x)}^T \lambda = -c(x)$.
Substituting these identities into the primal-dual gap yields
\begin{equation*}
    \begin{array}{rcl}
    P-D
    &=& {c(x)}^T u + {b(x)}^T \lambda\\
    &=& {c(x)}^T u + {(A(x)u + s)}^T \lambda\\
    &=& {c(x)}^T u + u^T {A(x)}^T \lambda + s^T \lambda\\
    &=& s^T \lambda.
    \end{array}
\end{equation*}
Because $(u,s)$ is primal feasible and $\lambda$ is dual feasible, weak
duality gives
\begin{equation*}
    D \leq q^\star(x) \leq P,
\end{equation*}
where $q^\star(x)$ is the optimal value of the canonical lower problem (\ref{prob:lower_cone}) for the given parameter $x$.
It then follows that
\begin{equation*}
    0 \leq P - q^\star(x) \leq P - D = s^T \lambda.
\end{equation*}
According to the losslessness assumption in \S\ref{sec:asps}, we have
\begin{equation*}
    p^\star(x) = q^\star(x).
\end{equation*}
Moreover, the feasible point recovery property ensures that $\rho(u)$ is feasible for the original lower problem and that
\begin{equation*}
    f_0(x, \rho(u)) \leq P.
\end{equation*}
Lower-level feasibility also gives
\begin{equation*}
    f_0(x, \rho(u)) \geq p^\star(x).
\end{equation*}
Combining these inequalities, we have
\begin{equation*}
    0 \leq f_0(x, \rho(u)) - p^\star(x)
    \leq P - p^\star(x)
    \leq s^T \lambda.
\end{equation*}
Finally, feasibility of (\ref{prob:approx_single_level}) imposes $s^T \lambda \leq \epsilon$, which proves (\ref{eq:lower_certificate}).
\end{proof}

\subsection{Proof of corollary~\ref*{cor:strongly_convex}}\label{sec:strongly_convex_proof}
\begin{proof}
The feasible set of the lower problem is convex because the lower problem is DPP-compliant.
Strong convexity on this set implies that the lower minimizer is unique:
If two distinct feasible points were both optimal, their midpoint would be feasible and, by strong convexity, would have a strictly smaller objective value, yielding a contradiction.
We may therefore write
\begin{equation*}
    S(x)=\{y^\star(x)\},
\end{equation*}
which is a singleton set containing the unique lower minimizer $y^\star(x)$.

For the fixed $x$, let $(u,s,\lambda)$ be any feasible lift in (\ref{prob:approx_single_level}) and set $y=\rho(u)$.
By the feasible-point recovery property in \S\ref{sec:asps}, $y$ is feasible for the original lower problem.
For any $\theta \in (0,1)$, convexity of the feasible set implies that
\begin{equation*}
    y_\theta = (1 - \theta) y^\star(x) + \theta y
\end{equation*}
is lower feasible.
By optimality of $y^\star(x)$ and $\mu$-strong convexity,
\begin{equation*}
    \begin{array}{rcl}
        f_0(x, y^\star(x)) &\leq& f_0(x, y_\theta)\\[5pt]
        &\leq& \displaystyle{(1 - \theta) f_0(x, y^\star(x)) + \theta f_0(x, y) - \frac{\mu}{2} \theta(1 - \theta) \norm{y - y^\star(x)}_2^2}.
    \end{array}
\end{equation*}
After rearranging and dividing by $\theta>0$, we obtain
\begin{equation*}
    f_0(x, y) - p^\star(x) \geq \frac{\mu}{2}(1 - \theta) \norm{y - y^\star(x)}_2^2.
\end{equation*}
Taking $\theta \to 0$ yields
\begin{equation*}
    f_0(x, \rho(u)) - p^\star(x) \geq \frac{\mu}{2} \norm{\rho(u) - y^\star(x)}_2^2.
\end{equation*}
Recall that proposition~\ref{prop:lower_certificate} bounds the left-hand side by
$\epsilon$.
Rearranging yields (\ref{eq:lower_distance_bound}).
\end{proof}

\section{Consistency of globally solved relaxations}\label{sec:gap_continuation_consistency}
For $\epsilon \geq 0$, let $\mathcal{W}_\epsilon$ denote the feasible set of the relaxed problem (\ref{prob:approx_single_level}) in the full lifted space of $w = (x,u,s,\lambda)$.
In particular, $\mathcal{W}_0$ is the feasible set of the exact single-level reformulation (\ref{prob:single_level}) (\ie, when $\epsilon = 0$), since $s \in \cK$ and $\lambda \in \cK^*$ imply $\lambda^T s \geq 0$.
For the result below, we additionally assume that the exact reformulation (\ref{prob:single_level}) is feasible and that $\mathcal{W}_{\tilde{\epsilon}}$ is compact for some $\tilde{\epsilon} > 0$.
The compactness of $\mathcal{W}_{\tilde{\epsilon}}$ is only a sufficient condition and may be replaced by any application-specific condition ensuring that the sequence of lifted solutions introduced below has an accumulation point.

\begin{proposition}\label{prop:gap_continuation}
    \emph{Consistency of globally solved relaxations.}
    Suppose that the assumptions in \S\ref{sec:asps} and the lifted-compactness condition above hold.
    Let the sequence $(\epsilon^{(k)})$ satisfy $0 < \epsilon^{(k)} \leq \tilde{\epsilon}$ and decrease to zero for $k = 0, 1, 2, \ldots$, and let
    \begin{equation*}
        w^{\star(k)} = (x^{\star(k)}, u^{\star(k)}, s^{\star(k)}, \lambda^{\star(k)})
    \end{equation*}
    be a global solution of (\ref{prob:approx_single_level}) with $\epsilon = \epsilon^{(k)}$.
    Then every accumulation point of the sequence $(w^{\star(k)})$ is a global solution of the exact reformulation (\ref{prob:single_level}), and the optimal values of the relaxed problems converge to the optimal value of the exact reformulation.
\end{proposition}

\begin{proof}
Because $0 < \epsilon^{(k)} \leq \tilde{\epsilon}$, the nesting of the relaxed feasible sets gives
\begin{equation*}
    w^{\star(k)} \in \mathcal{W}_{\epsilon^{(k)}} \subseteq \mathcal{W}_{\tilde{\epsilon}}.
\end{equation*}
The lifted-compactness condition therefore ensures that the sequence $(w^{\star(k)})$ has an accumulation point.
Let $\bar{w} = (\bar{x}, \bar{u}, \bar{s}, \bar{\lambda})$ be an arbitrary accumulation point, and choose a subsequence such that
\begin{equation*}
    w^{\star(k_j)} \to \bar{w}.
\end{equation*}

The upper constraints are closed by continuity of $F_i$ for $i = 1, \ldots, m$ and $\rho$, the primal- and dual-feasibility equations are continuous, and the cones $\cK$ and $\cK^*$ are closed.
Hence, all constraints of (\ref{prob:single_level}) other than exact complementarity are preserved in the limit.
Moreover, feasibility of each relaxed problem gives
\begin{equation*}
    0 \leq {\lambda^{\star(k_j)}}^T s^{\star(k_j)} \leq \epsilon^{(k_j)} \to 0.
\end{equation*}
Continuity of the inner product then yields
$\bar{\lambda}^T \bar{s} = 0$, so $\bar{w} \in \mathcal{W}_0$.

Let $w \in \mathcal{W}_0$ be arbitrary.
Because $\mathcal{W}_0 \subseteq \mathcal{W}_{\epsilon^{(k_j)}}$, global optimality of $w^{\star(k_j)}$ for its relaxed problem gives
\begin{equation*}
    F_0(x^{\star(k_j)}, \rho(u^{\star(k_j)})) \leq F_0(x, \rho(u)).
\end{equation*}
Passing to the limit and using continuity of $F_0$ and $\rho$ gives
\begin{equation*}
    F_0(\bar{x}, \rho(\bar{u})) \leq F_0(x, \rho(u)).
\end{equation*}
Thus, $\bar{w}$ is globally optimal over $\mathcal{W}_0$.
Because the accumulation point was arbitrary, every accumulation point of $(w^{\star(k)})$ is a global solution of (\ref{prob:single_level}).

Finally, let $v^{\star(k)}$ denote the optimal value of the relaxed problem with tolerance $\epsilon^{(k)}$, and let $v_0^\star$ denote the optimal value of the exact reformulation when $\epsilon = 0$.
Since the feasible sets shrink as $\epsilon^{(k)}$ decreases and $\mathcal{W}_0 \subseteq \mathcal{W}_{\epsilon^{(k)}}$, we have
\begin{equation*}
    v^{\star(k)} \leq v^{\star(k+1)} \leq v_0^\star.
\end{equation*}
Along the convergent subsequence above, continuity of the objective and global optimality of $\bar{w}$ give $v^{\star(k_j)} \to v_0^\star$.
The monotonicity of $(v^{\star(k)})$ then implies that the full sequence converges to $v_0^\star$.
\end{proof}

Under the assumptions in \S\ref{sec:asps}, the conic KKT conditions (\ref{eq:kkt_conic}) are necessary and sufficient for $\rho(u)$ to solve the lower problem for a given $x$.
Thus, the recovery map $(x, u, s, \lambda) \mapsto (x, \rho(u))$ sends every feasible point of the exact reformulation (\ref{prob:single_level}) to a feasible point of the original bilevel problem (\ref{prob:blv}), while every bilevel-feasible point admits such a lift; the upper objective is unchanged under this correspondence.
Consequently, every accumulation point identified by proposition~\ref{prop:gap_continuation} recovers a globally optimal solution of the original bilevel problem.

Note that without global solves (\eg, in algorithm~\ref{alg:gap_cont} and practical applications), any accumulation point of a sequence of exactly feasible relaxed points with $\epsilon^{(k)} \to 0$ is still feasible for the exact reformulation, but it need not be optimal.
For locally or approximately solved relaxations, the lower-level certificate in proposition~\ref{prop:lower_certificate} remains valid whenever primal feasibility, dual feasibility, cone membership, and the relaxed complementarity constraint are satisfied, but convergence to a globally optimal bilevel solution is not guaranteed.

\newpage
\section*{Acknowledgments}
This work has been funded as part of BrainLinks-BrainTools, which is funded by the Federal Ministry of Economics, Science and Arts of Baden-W\"urttemberg within the sustainability program for projects of the Excellence Initiative II, and CRC/TRR 384 ``IN-CODE''.

\newpage
\bibliography{refs}

\end{document}